\documentclass[11pt]{amsart}

\usepackage{amsmath}
\usepackage{amssymb}
\usepackage{graphicx}
\usepackage{tikz}
\usetikzlibrary{arrows.meta,positioning}
\usepackage{cite}
\usepackage[hidelinks]{hyperref}

\newtheorem{theorem}{Theorem}[section]
\newtheorem{corollary}[theorem]{Corollary}
\newtheorem{lemma}[theorem]{Lemma}
\newtheorem{proposition}[theorem]{Proposition}
\theoremstyle{definition}
\newtheorem{definition}[theorem]{Definition}
\newtheorem{example}[theorem]{Example}
\newtheorem{remark}[theorem]{Remark}

\numberwithin{equation}{section}

\title[Hyperbolic Monge-Ampère Equation on a Cylinder]{The Hyperbolic Monge-Amp\`ere Equation on a Cylinder: Well-Posedness and Stability} \thanks{This paper is dedicated to Professor Peter Constantin on the occasion of his 75th birthday with gratitude for his generosity, mentorship, guidance, and support.}

\author[M. Deliyianni]{Maria Deliyianni}

\address[M. Deliyianni]{Department of Mathematics, University of California, Riverside, 900 University Ave., Riverside, CA 92521}
\email{{\tt maria.deliyianni@ucr.edu}}

\author[S. C. Venkataramani]{Shankar C. Venkataramani} 
\address[S. C. Venkataramani]{Department of Mathematics, University of Arizona, 617 N. Santa Rita Ave., Tucson, AZ 85721}
\email{\tt shankar@arizona.edu}



\date{\today}

\begin{document}

\begin{abstract}
This paper develops a rigorous analytic framework for the hyperbolic Monge--Ampère equation on strip-like domains, which model wrinkled patterns in thin elastic sheets. We address the rigid side of the rigidity--flexibility dichotomy, defining it by partial convexity rather than high smoothness. The hodograph transformation, which is valid under partial convexity, converts the nonlinear Monge--Ampère equation into a linear damped wave equation, yielding a well-posed Cauchy--Goursat problem in hodograph variables. The key challenge is the corner singularity where characteristic and non-characteristic boundary data meet. To resolve it, we develop a parametrix--corrector decomposition that models the solution's corner singular behavior. This recasts the problem as a singular Volterra equation, for which we prove existence and uniqueness of a new class of hodograph weak solutions. Reconstructing a physical solution additionally requires sufficient regularity and invertibility of the hodograph map; an explicit fold example shows that this can fail. Finally, we derive energy estimates, including a conditional linearized estimate under curvature perturbations for finite `bare' energy solutions, and propagation and secant estimates for parametrix--corrector solutions with regular correctors of finite `renormalized' energy.
\end{abstract}

\maketitle

\tableofcontents

\section{Introduction}\label{sec:intro}

The intricate, frilly patterns of leaves, fungi, and marine invertebrates are striking examples of hyperbolic geometry in nature \cite{Sharon2004Leaves,sharon2007geometrically, Plant2}. %
These complex, multi-scale wrinkled shapes arise in thin elastic sheets when non-uniform growth or stress induces an intrinsic metric with negative Gaussian curvature \cite{sharon2007geometrically}. The mathematical description of these surfaces often leads to the Hyperbolic Monge--Ampère equation, the prescribed Gauss curvature equation for isometric immersions of surfaces with negative curvature \cite{poznyak1973isometric,stoker}. While central to differential geometry and nonlinear analysis, in our context this equation presents a fundamental puzzle--the {\em rigidity-flexibility dichotomy} \cite{Gemmer2016Isometric}.

On one hand, classical results in differential geometry establish that smooth ($C^2$) isometric immersions of hyperbolic surfaces are rigid \cite{efimov1964generation}; isometric immersions of large domains necessarily develop large curvatures and incur a prohibitive energetic cost. This mathematical rigidity seems to contradict the existence of the floppy, intricately wrinkled structures seen in nature \cite{Sharon2004Leaves,Yamamoto2021Nature}. Our recent work has shed some light on this issue by considering a ``flexible'' regime of less-smooth surfaces \cite{Gemmer2016Isometric,gemmer2013transitions, Shearman2021Distributed, Yamamoto2021Nature}. By introducing topological defects known as {\em branch points}, surfaces of class $C^{1,1}$ can form complex, lower-energy wrinkling patterns, providing an avenue for flexibility that nature appears to exploit \cite{Yamamoto2021Nature}.

Our work is just one aspect of the ``flexible" side of the dichotomy. Flexibility itself has been revealed to be a spectrum of behaviors that depends on the precise regularity of the surface. At one extreme lies the foundational Nash-Kuiper theorem on $C^1$ isometric embeddings~\cite{ Kuiper1955C1,Nash1954C1}, which established a paradigm of maximal flexibility by constructing an abundance of highly wrinkled solutions to the isometric immersion problem. The modern exploration of this spectrum, particularly on the intermediate $C^{1,\alpha}$ H\"older scale, has been driven by the theory of {\em convex integration} \cite{Eliashberg2002book,Gromov1986book,Spring2010book}. This framework has been successfully applied to the isometric immersion problem and the associated Monge–Ampère system by a community of researchers including Borisov~\cite{Borisov2004irregular}, Conti, De Lellis, and Sz\'ekelyhidi \cite{Conti2012h-principle} and Lewicka and Pakzad~\cite{ LewickaMonge2,LewickaMonge3, LewickaMonge1}, leading to the construction of families of ``wild" $C^{1,\alpha}$ isometries that are candidates for weak solutions of the associated Monge--Amp\`ere system \cite{Cao2,Cao1}.

Within this spectrum of flexibility, our prior work has identified the $C^{1,1}$ class as particularly relevant to physical systems. Unlike the less regular $C^{1,\alpha}$ solutions ($\alpha<1$), $C^{1,1}$ surfaces possess a well-defined bending energy \cite{lewicka2011scaling}, providing a potential energetic selection principle for the specific patterns observed in nature. Just as $C^{1,1}$ surfaces are, in a sense, ``borderline" flexible, our work in this paper looks at the complementary problem of characterizing hyperbolic surfaces that are borderline rigid. We focus on the rigid regime by analyzing the structure and stability of {\em partially convex} solutions, where flexibility is suppressed. The manifestation of rigidity in highly smooth ($C^3$) solutions is a transport law that dictates that the wrinkling wavelength cannot change in the transverse direction. This law rules out the possibility of multi-generational sub-wrinkling and implies that a curvature component must maintain a definite sign between lines of inflection. We interpret this transport law not as an impediment, but as the very definition of the rigid solutions we seek to construct. The central goal of this paper is to characterize this regime by replacing the overly restrictive $C^3$ smoothness condition with the more fundamental property that defines rigidity in this setting: {\em partial convexity}. Classical $C^3$ solutions exhibit a rigid structure composed of regions of partial convexity (where $w_{xx}$ has a definite sign), separated by non-bifurcating lines of inflection. Our approach formalizes the analysis on an individual patch and identifies the compatibility required to glue symmetric patches smoothly.

To analyze this structure, we study periodic solutions on a cylindrical domain $S^1 \times I$, where $I \subseteq \mathbb{R}$ is an interval (possibly unbounded), $x$ is a coordinate along the compact periodic direction (the fibers) and the transverse coordinate, $y$, acts as a time-like variable. The problem can thus be viewed as a ``1+1" dimensional hyperbolic system describing the evolution of periodic profiles in ``time.'' In this context, the {\em hodograph transformation}, which exchanges the spatial variable $x \in S^1$ with its corresponding slope $p = w_x$, is the natural analytical tool, as its validity is predicated on the very condition of partial convexity ($w_{xx} \neq 0$). This transformation converts the nonlinear Monge--Ampère equation,
\begin{equation}\label{eq:hma}
    w_{xx}w_{yy} - w_{xy}^2 = -\lambda(y)^2,
\end{equation}
into a linear damped wave equation \cite{CourantHilbert1962}. This structure allows us to define our central object of study: {\em hodograph weak solutions}. This notion, which is based on an integral identity in the transformed $(p,y)$ hodograph plane, is distinct from the classical Alexandrov theory for elliptic problems \cite{Alexandrov}. It provides a weak solution concept for the transformed problem on a candidate partially convex patch. Returning to physical variables is a separate inversion step and requires the hodograph map to remain sufficiently regular and nondegenerate.

This framework leads naturally to the formulation of a {\em Cauchy--Goursat problem} \cite{CourantHilbert1962}. However, a key difficulty arises from the geometry of this problem,  the {\em corner singularity} where non-characteristic Cauchy data meet characteristic Goursat data. As we argue below, standard notions of $C^k$ compatibility for the boundary data are insufficient. Indeed, even a solution that is perfectly smooth in the original $(x, y)$ variables will naturally generate inverse square-root singularities in its derivatives on the boundary of the hodograph domain, corresponding to cusps. We therefore introduce a more suitable notion of {\em weak compatibility}, which correctly captures the minimal regularity required for a well-posed Cauchy-Goursat problem.

Our primary contribution is the resolution of this singular problem, which we analyze in two distinct settings. For a finite cylinder, the problem is regularized and shown to be equivalent to a standard Volterra integral equation with a bounded kernel, for which well-posedness follows from classical theory \cite{TricomibookIntegralEquations}. The more challenging case is the semi-infinite cylinder, which corresponds to a pure Goursat problem with discontinuous boundary data at the origin. Herein lies the main technical novelty of our work: the development of a parametrix--corrector decomposition. The parametrix accounts for the discontinuity in the Goursat data and reduces the problem to a Volterra equation for a more regular corrector term. Its kernel has degree $-1$ under joint dilation toward the corner, so the bounded-kernel estimates from the finite-cylinder problem are not uniform there. An exact identity for the integrated kernel masses nevertheless yields a contraction under the stated monotonicity assumptions. By subtracting the model singular behavior, the decomposition also permits compatibility conditions to be imposed on the \textit{corrector} rather than on the full solution. Finally, we construct an energy framework giving a conditional first-variation estimate for finite-energy solutions under perturbations of $\lambda(y)$; see Remark~\ref{rem:energy-scope} for the separate regularity class used for renormalized singular states.

\label{para:novelty}
\textbf{The principal contributions of this paper are the following.}
(i) We introduce the notion of \emph{hodograph weak solutions}
(Definition \ref{defn:hodograph-weak}) for the transformed hyperbolic
Monge--Amp\`ere problem, a concept adapted to the symmetries of the
hodograph plane and distinct from the Alexandrov and viscosity notions
commonly used in the elliptic setting; see
Remark \ref{rem:weak_sol_context} for a discussion of these distinctions.
(ii) We develop a \emph{singular Volterra formulation} that recasts the
Cauchy--Goursat problem as a two-dimensional Volterra integral equation
and proves transformed existence and uniqueness.
(iii) The central technical device is the
\emph{parametrix--corrector decomposition} of
\S\ref{sec:parametrix}: the parametrix $x_p$ (Eq.~\eqref{eq:x_p})
models the square-root corner behavior suggested by
\S\ref{sec:corner-singularity}, reducing the problem to a regular
corrector governed by a Volterra equation with a degree $-1$
corner-scale kernel.
(iv) We derive forward energy estimates and a conditional linearized
estimate for finite-energy solutions under curvature perturbations
(\S\ref{sec:estimates}), together with fixed-$\lambda$ secant estimates
for a class of solutions with finite \emph{renormalized} energies.

The diagram below summarizes how these contributions fit into the
analytic route from the original geometric problem to the construction
and analysis of hodograph weak solutions.

\begin{center}\label{dia:roadmap}
\begin{tikzpicture}[
  stage/.style={draw=blue!70!black, rounded corners=3pt, fill=blue!7,
    minimum height=1.35cm, text width=3.0cm, align=center, font=\small},
  critical/.style={stage, draw=blue!80!black, fill=blue!17},
  result/.style={stage, draw=blue!60!black, fill=blue!25},
  flow/.style={-{Latex[length=2.2mm]}, blue!70!black, line width=.7pt}
]
\node[stage] (geometry) {\textbf{Physical patch}\\[-1pt]
  $\det D^2w=-\lambda(y)^2$\\[-1pt]
  partial convexity};
\node[stage, right=9mm of geometry] (hodograph) {\textbf{Hodograph coordinates}\\[-1pt]
  $p=w_x$\\[-1pt]
  inverse: $x=x(p,y)$\\[-1pt]
  $L[x]=0$};
\node[critical, right=9mm of hodograph] (corner) {\textbf{Corner data}\\[-1pt]
  Cauchy--Goursat problem\\[-1pt]
  square-root trace};
\node[critical, below=11mm of corner] (parametrix) {\textbf{Singular resolution}\\[-1pt]
  parametrix + corrector\\[-1pt]
  Volterra equation};
\node[result, left=9mm of parametrix] (result) {\textbf{Hodograph solutions}\\[-1pt]
  fold-free inversion\\[-1pt]
  energy estimates\\[-1pt]
  renormalized energy};
\draw[flow] (geometry) -- (hodograph);
\draw[flow] (hodograph) -- (corner);
\draw[flow] (corner) -- (parametrix);
\draw[flow] (parametrix) -- (result);
\end{tikzpicture}

\smallskip
\small\emph{Analytic route from a partially convex physical patch to hodograph solutions and, when invertible, back to physical variables.}
\end{center}

The paper is organized as follows. In Section \ref{sec:problem_and_rigidity}, we define the geometric setting of the problem and derive the transport law that demonstrates the inherent rigidity of classical solutions. Section \ref{sec:hodograph} introduces our primary analytical tool, the hodograph transformation, which linearizes the Monge-Ampère equation. In Section \ref{sec:symmetry}, we leverage solution symmetries to reduce the problem to a fundamental domain and formally set up the Cauchy-Goursat problem. The well-posedness of this transformed problem is the focus of Section \ref{sec:cauchy-goursat}, where we establish existence for the finite cylinder before developing a parametrix--corrector decomposition to resolve the boundary singularity that arises in the more challenging semi-infinite cylinder case in Section~\ref{sec:singularities}. In Section \ref{sec:estimates}, we derive forward, linearized, and fixed-curvature secant energy estimates in their stated regularity classes. We conclude in Section \ref{sec:future} with the geometric \textit{``arrow of time''} for the mixed problem and directions for future work.

\section{Problem Formulation and Rigidity}
\label{sec:problem_and_rigidity}

\subsection{A Running Example: Product Solution and the Exponential Curvature Function}
\label{subsec:geometric_setting}
We begin with a running example for the model equation itself. Let
\begin{equation}\label{eq:running_product_solution}
    w_0(x,y)=e^{-y}\cos x.
\end{equation}
Then
\[
    \det D^2w_0=-e^{-2y},
\]
so $w_0$ solves \eqref{eq:hma} with $\lambda(y)=e^{-y}$. This elementary solution already displays the three features that organize the analysis below:
\begin{equation}\label{eq:running_product_derivatives}
    (w_0)_{xx}=-e^{-y}\cos x,
    \qquad p=(w_0)_x=-e^{-y}\sin x.
\end{equation}
On the fundamental patch $|x|<\pi/2$, $(w_0)_{xx}<0$, while the fixed lines $x=\pm\pi/2$ are inflection lines. Thus the example is \emph{partially convex} on this patch and its \emph{wavelength does not refine} as $y$ varies. Moreover, setting $u=e^{-y}$, the inverse slope map on this patch is
\[
    x=\arcsin\!\left(-\frac{p}{u}\right),
    \qquad x_p=-\frac{1}{\sqrt{u^2-p^2}}.
\]
Although $w_0$ is smooth in physical coordinates, the inverse hodograph map $(p,y)\mapsto(x,y)$ has square-root cusp behavior as $p\to\pm u$, precisely at the inflection lines, and $x_p$ develops inverse-square-root singularities there. This is the local mechanism behind the singular boundary traces treated later in the Cauchy--Goursat problem (\S \ref{sec:cauchy-goursat} below).

The above discussion motivates the general class of strip-like geometries we study. Metrics of the form
\[
    ds^2 = (1+f(y))^2\,dx^2 + dy^2
\]
capture these same features. The Gaussian curvature for this metric is $K(y) = -f''(y)/(1+f(y))$. In the small deformation limit where $|f| \ll 1$, this simplifies to $K \approx -f''(y)$. We thus define the curvature via
\[
    K(y) = -\lambda(y)^2 \quad \text{with } \quad \lambda(y)^2 := f''(y),
\]
referring to $\lambda$ as the \emph{curvature function}. With this notation, the immersion function $w(x,y)$ satisfies the hyperbolic Monge--Ampère equation \eqref{eq:hma}. We focus on solutions that are periodic in the $x$-direction, which naturally leads to considering domains that are topologically cylinders, such as $S^1 \times [l,L]$ or $S^1 \times [0,\infty)$.

\begin{remark}[Standing assumptions on $\lambda$]\label{rem:lambda_assumptions}
For convenience, we collect here the standing assumptions on the curvature function $\lambda(y)$ used throughout this work:
\begin{enumerate}
    \item[(1)] Regularity and sign: $\lambda \in C^1$ with $\lambda(y) > 0$.
    \item[(2)] Integrability: $u(y) = \int_y^\infty \lambda(\eta)\,d\eta < \infty$, which provides the natural timelike coordinate on cylindrical domains.
    \item[(3)] Asymptotic decay for stability: for the stability estimates of Section~\ref{sec:estimates}, we require that the logarithmic derivative satisfies $\lambda_y(y)/\lambda(y) \leq -m_0 < 0$ for all sufficiently large $y$.
\end{enumerate}
The singular-corner construction in Section~\ref{sec:singularities} imposes additional assumptions on the extension of $\lambda$ to $u=0$; these are stated in Lemma~\ref{lem:regularity-residual} and Theorem~\ref{thm:corrector}.
\end{remark}

\subsection{Classical Solutions and Rigidity}
\label{subsec:classical_rigidity}
Before introducing the hodograph framework, it is important to understand the rigidity imposed by the Monge--Ampère equation on classical solutions. A natural starting point is to seek $C^3$ (or smoother) solutions $w(x,y)$. However, the structure of the equation imposes a strong constraint that prevents refinement of oscillatory behavior, ruling out physically observed phenomena \cite{Sharon2004Leaves,sharon2002nature, sharon2007geometrically}. This rigidity can be seen by differentiating the Monge--Ampère equation \eqref{eq:hma} with respect to $x$, which yields the conservation law:
\begin{equation}
    w_{xxx}w_{yy} + w_{xx}w_{xyy} -2w_{xxy}w_{xy}  = 0.
    \label{eq:xvar}
\end{equation}
This equation implies a remarkable property: the level sets of the curvature component $w_{xx}$ are transported along a vector field that is transverse to the fibers $y = \text{constant}$. To see this, define a vector field $\mathbf{v}$ by
\[
\mathbf{v} = -\frac{w_{xy} w_{yy}}{2 \lambda(y)^2} \frac{\partial}{\partial x} + \frac{\partial}{\partial y}.
\]
A direct computation of $\mathbf{v}\cdot \nabla w_{xx}$, in conjunction with \eqref{eq:xvar}, yields
\begin{align}
\mathbf{v} \cdot \nabla(w_{xx}) & = \left( \frac{w_{xy}w_{xyy} - 2w_{yy}w_{xxy}}{2\lambda(y)^2} \right) w_{xx}.
\label{eq:curvature_transport}
\end{align}
This demonstrates that $\mathbf{v}\cdot \nabla(w_{xx}) = \alpha(x,y) w_{xx}$ for some function $\alpha(x,y)$. In particular, the sign of $w_{xx}$ does not change along the orbits of $\mathbf{v}$. This transport law reveals an inherent rigidity, which precludes multi-generational wrinkling and ensures that partial convexity is preserved as the solution evolves in the time-like variable $y$. This motivates the development of an analytical framework specifically designed to construct and characterize solutions within the rigid regime.

\begin{remark}[Geometric meaning of partial convexity]\label{rem:partial_convexity_geom}
A solution $w(x,y)$ is \emph{strictly partially convex} on a patch if $w_{xx}$ has a fixed nonzero sign. Equivalently, $p=w_x$ is strictly monotone in $x$, and the hodograph map $(x,y)\mapsto(p,y)$ is a local diffeomorphism. Geometrically, this means the surface has a definite ``curvature component'' in the $x$-direction and no inflection in the interior of the patch. The transport law \eqref{eq:curvature_transport} shows that an initially nonzero value of $w_{xx}$ cannot change sign along its transported flow line while the classical solution remains in the patch. The product solution $w_0 = e^{-y} \cos(x)$ is a simple instance.
\end{remark}
 
\section{The Hodograph Formulation}
\label{sec:hodograph}

The rigidity of classical solutions, as demonstrated by the transport law \eqref{eq:curvature_transport} in the physical $(x, y)$ coordinates, motivates a change of framework. To analyze the structure of partially convex solutions, we employ the {\em hodograph transformation}. This is a classic technique for linearizing certain nonlinear hyperbolic systems by interchanging dependent and independent variables. Examples include its application to {\em gas dynamics} and to the {\em Born-Infeld equation} from nonlinear electrodynamics, both of which are treated in detail by Whitham~\cite{whitham1974linear} (Chapters 6 and 17, respectively). In our setting, this transformation similarly exchanges $x$ (corresponding to the compact `spatial' direction)  with slope variables $p = w_x$, converting the fully nonlinear Monge–Ampère equation into a linear damped wave equation in characteristic coordinates.

The analysis is simplified by reparametrizing the transverse coordinate interpreted as time. We introduce the new variable $u$ defined in terms of the curvature function $\lambda(y)$:
\begin{equation}\label{eq:defofu}
    u(y) = \int_y^\infty \lambda(\eta)\, d\eta.
\end{equation}
In the case of an infinite cylinder, the far-field limit $y \to \infty$ corresponds to $u \to 0$. For convenience, we will often treat $\lambda$ as a function of $u$, denoting the composition $\lambda(y(u))$ simply as $\lambda(u)$. To avoid ambiguity in what follows, we adopt the convention that the prime notation refers exclusively to differentiation with respect to $u$, so that 
$$
\lambda' \equiv \frac{d\lambda}{du} = -\frac{1}{\lambda} \frac{d \lambda}{dy}  = - \frac{d}{dy} (\log \lambda).
$$ 

\subsection{Exterior Differential System and Characteristics}
\label{subsec:eds}

The language of Exterior Differential Systems (EDS) provides a powerful geometric approach to analyzing PDEs, where the equation is recast as a system of differential forms on a jet bundle~\cite{EDS, Ivey2003book}. Fundamental results in this area, such as the Cartan--Kähler theorem, constitute a ``formal theory" that can guarantee the existence of local, analytic solutions under certain involutivity conditions \cite{Baker2002jetbundles}. However, this formal theory does not by itself address the global existence and uniqueness for a specific boundary value problem like the one we study. Our aim here is therefore not to apply the full Cartan--Kähler machinery, but to use the geometric insight from the EDS framework to derive an equivalent linear PDE. This linear equation is the key to the rigorous, global analysis that follows.

To formulate the hyperbolic Monge--Amp\`ere equation as an EDS, we consider the first jet bundle, $J^1(\Omega)$, over the domain $\Omega \subset \mathbb{R}^2$. This space formalizes the geometry by including not only the base coordinates but also the function's value and its first derivatives. A point in $J^1(\Omega)$ is given by the coordinates $(x, y, w, p, q)$, which correspond to the location $(x, y) \in \Omega$, the value of the immersion $w: \Omega \to \mathbb{R}$, and the partial derivatives $p = w_x$ and $q = w_y$.
The EDS for the Monge--Ampère equation is generated by the contact 1-form $\theta$ and the 2-form $\omega$ \cite{Ivey2003book}:
\begin{align}
    \theta &= dw - p\,dx - q\,dy \nonumber\\
    \omega &= dp \wedge dq + \lambda^2 dx \wedge dy.
    \label{eq:MA-EDS}
\end{align}
To adapt the system to characteristic variables, we recall from \eqref{eq:defofu} that the differential relationship is $du = -\lambda\,dy$. With this change of variables, the forms can be combined into the following identity, which reveals the characteristic structure of the system:
\[
\omega \pm \lambda(u) d\theta = (dp \mp du) \wedge (dq \mp \lambda(u)\, dx).
\]

This factorization shows that the original system yields two characteristic subsystems \cite{Ivey2003book}. To make this structure explicit, we introduce the characteristic coordinates $s = \tfrac{1}{2}(u+p)$ and $t = \tfrac{1}{2}(u-p)$. In these variables, the subsystems are defined by the Pfaffian equations with $\lambda(u) = \lambda(s+t)$:
\begin{align}
 dt &= 0, \quad dq - \lambda(u)\,dx = 0, \label{csystem1}\\[0.5em]
 ds &= 0, \quad dq + \lambda(u)\,dx = 0.\label{csystem2}
\end{align}

From these relations, we can extract a more compact formulation that will be central in what follows. Specifically, along any path we have 
\[
dq = \lambda(u) (\mu ds - \nu dt), \quad dx = \mu ds + \nu dt,
\]
where we define $\mu = \tfrac{\partial x}{\partial s} = x_s$ and $\nu = \tfrac{\partial x}{\partial t} = x_t.$ Since $q$ is a function of $s$ and $t$, its exterior derivative must vanish, namely $d(dq)=0$. This implies
\[
d\!\left( \lambda(u) \left( \frac{\partial x}{\partial s}\, ds - \frac{\partial x}{\partial t}\, dt \right) \right) = 0.
\]

Thus, the transformed equation requires the 1-form
\begin{equation}\label{eq:one_form_alpha}
\alpha = \lambda(s+t) (x_s\, ds - x_t\, dt)
\end{equation}
is closed ($d\alpha=0$), hence locally exact. It is globally exact on the simply connected hodograph domains used below. On such a domain, Stokes' theorem gives the equivalent contour condition $\oint_C\alpha=0$ for every closed contour $C$.

\subsection{The Linear PDE from the Hodograph Transformation}
\label{subsec:linear_pde}
For the 1-form $\alpha = P\,ds + Q\,dt$ to be closed, its components must satisfy the compatibility condition
\[
\frac{\partial Q}{\partial s} = \frac{\partial P}{\partial t}.
\]
From \eqref{eq:one_form_alpha} we have $P = \lambda x_s$ and $Q = -\lambda x_t$. Denoting $\lambda=\lambda(s+t)$ and $\lambda' = \tfrac{d\lambda}{d(s+t)}$, we compute:
\begin{align*}
    \frac{\partial P}{\partial t} &= \frac{\partial}{\partial t}\big(\lambda x_s\big) 
    = \lambda' x_s + \lambda x_{st}, \\[0.5em]
    \frac{\partial Q}{\partial s} &= \frac{\partial}{\partial s}\big(-\lambda x_t\big) 
    = -\lambda' x_t - \lambda x_{st}.
\end{align*}

Equating the two derivatives gives
\begin{equation}\label{eq:LinearMA}
-\lambda' x_t - \lambda x_{st} = \lambda' x_s + \lambda x_{st}.
\end{equation}
Let us introduce the linear operator 
\begin{equation}\label{eq:LinearOperator}
L[x] := 2\lambda\,x_{st} + \lambda'(x_s + x_t).
\end{equation}
Rearranging terms in \eqref{eq:LinearMA} leads to the linear PDE for $x(s,t)$:
\begin{equation}
L[x] = 0.
\label{eq:linear_pde_st}
\end{equation}

If $\lambda \neq 0$, we can also write this as:
\begin{equation}
\label{wave-equation}
x_{st} + \frac{\lambda'}{2\lambda} (x_s + x_t) = 0
\end{equation}
This PDE describes %
functions $x=x(s,t)$ satisfying the requirement that the one-form $\alpha$ is closed (and therefore locally exact).

\subsection{Non-integrability and the Goursat Problem}
\label{subsec:non_integrable}
At first glance, one might expect the situation to parallel the wave equation: once the characteristic structure is revealed, one could integrate along characteristics to obtain explicit formulas for the solution, as in d'Alembert’s formula. A direct obstruction is that the characteristic Pfaffian subsystems fail the Frobenius condition unless $\lambda$ is constant, as the following lemma shows.

\begin{lemma}
The characteristic Pfaffian subsystems \eqref{csystem1} and \eqref{csystem2} are Frobenius integrable if and only if the function $\lambda(u)$ is constant.
\end{lemma}

In what follows, $\beta_i$ and $\zeta_i$ ($i=1,2$) denote the two 1-forms in each characteristic Pfaffian subsystem: specifically $\beta_1 = dt$, $\zeta_1 = dq - \lambda(u)\,dx$ for subsystem~\eqref{csystem1}, and $\beta_2 = ds$, $\zeta_2 = dq + \lambda(u)\,dx$ for subsystem~\eqref{csystem2}.

\begin{proof}
For a characteristic subsystem to be Frobenius, $d\beta_i$ and $d\zeta_i$ must belong to the differential ideal generated by $\beta_i$ and $\zeta_i$ \cite{EDS}. Let us check this for subsystem 2. We have
\begin{align*}
d\beta_2 &= d(ds) = 0, \\[0.5mm]
d\zeta_2 &= d(dq + \lambda(u)\,dx) = d\lambda(u) \wedge dx = \lambda'(u)\,du \wedge dx.
\end{align*}
It is clear that $d\beta_2$ lies in the ideal. For $d\zeta_2$, Frobenius integrability requires
$d\zeta_2 \wedge \beta_2 \wedge \zeta_2 = 0.$ We compute:
\begin{align*}
d\zeta_2 \wedge \beta_2 \wedge \zeta_2 
&= \big(\lambda'(u)\,du \wedge dx\big) \wedge ds \wedge \big(dq + \lambda(u)\,dx\big) \\[0.5em]
&= \lambda'(u)\,du \wedge dx \wedge ds \wedge dq.
\end{align*}
Since $u = s+t$, we have $du = ds+dt$, and hence the above becomes
\[
\lambda'(s+t)\,(ds+dt) \wedge dx \wedge ds \wedge dq 
= \lambda'(s+t)\,dt \wedge dx \wedge ds \wedge dq.
\]
This $4$-form vanishes identically if and only if $\lambda'(u)\equiv0$, i.e.\ $\lambda(u)$ is constant. A similar calculation applies to~\eqref{csystem1}. When $\lambda$ is constant, the displayed exterior derivatives lie in the respective ideals, so both subsystems are Frobenius.
\end{proof}

\noindent This failure of Frobenius integrability obstructs the simplest characteristic integration used for a d’Alembert-type formula. We therefore turn from seeking explicit formulas to analyzing a boundary value problem. For a hyperbolic PDE expressed in characteristic coordinates, such as Equation~\eqref{wave-equation}, the {\em Goursat problem}, where data are prescribed on two intersecting characteristic curves, provides the natural framework. This formulation is particularly well-suited for our analysis because it can be converted directly into a Volterra integral equation.

\section{Symmetric Solutions and Patch Gluing}
\label{sec:symmetry}
While the hodograph formulation linearizes the Monge–Ampère equation, it applies only on domains where the solution is partially convex ($w_{xx} \ne 0$). 
For periodic solutions, it is of course impossible to have $w_{xx} \neq 0$ everywhere on a compact fiber $y =$ constant. Consequently, the domains of partial convexity appear as patches bounded by lines of inflection $w_{xx} =0$. 
To construct a global solution, one must analyze the solution on a fundamental patch and understand how to glue it to its neighbors. 
The properties of this gluing are determined by the symmetries of the solution.

In this section, we show how symmetry provides the geometric framework for formulating the problem. Certain exact solutions exhibit strong reflectional and translational symmetries. These symmetries fix the location of inflection lines and determine how adjacent solution patches can be joined. By analyzing this class of symmetric solutions, we obtain the natural boundary configuration that motivates the Cauchy--Goursat formulation studied in the next section.

\subsection{Symmetries and Inflection Lines} 
We are motivated by the structure of the exact product solution $w(x,y) = e^{-y}\cos(x)$ of \eqref{eq:hma}. 
This example suggests looking at a class of solutions $w(x,y)$ that share its key symmetries:
\begin{enumerate}
    \item \textbf{Even under reflection about $x=0$:} The solution is even in $x$,
    \[
    w(-x,y) = w(x,y).
    \]
    \item \textbf{Odd under translation by a half-period:} The solution is anti-symmetric under a shift by half its period, which we take to be $\pi$,
    \[
    w(x+\pi,y) = -w(x,y).
    \]
\end{enumerate}
The second condition implies periodicity with period $2\pi$. 
Taken together, these symmetries fix the location of the inflection lines.

\begin{lemma}
If $w(x,y)$ is $C^2$ and satisfies the two symmetry conditions above, then it has lines of inflection at $x = \pm \tfrac{\pi}{2}$.
\end{lemma}

\begin{proof}
The translational anti-symmetry implies a corresponding anti-symmetry for the second derivative, $w_{xx}(x+\pi,y) = -w_{xx}(x,y).$ The reflectional symmetry implies that $w_{xx}$ is even,
$w_{xx}(-x,y) = w_{xx}(x,y).$

Evaluating the first identity at $x=-\tfrac{\pi}{2}$ gives
\[
w_{xx}(\tfrac{\pi}{2},y) = -w_{xx}(-\tfrac{\pi}{2},y).
\]
From the second identity we know $w_{xx}(-\tfrac{\pi}{2},y) = w_{xx}(\tfrac{\pi}{2},y)$. Substituting this relation yields
\[
w_{xx}(\tfrac{\pi}{2},y) = -w_{xx}(\tfrac{\pi}{2},y).
\]
Hence $2w_{xx}(\tfrac{\pi}{2},y) = 0$, so $w_{xx}(\tfrac{\pi}{2},y) = 0$. By the even symmetry of $w_{xx}$, it must also vanish at $x=-\tfrac{\pi}{2}$.
\end{proof}

\subsection{Smooth Gluing of Symmetric Solutions}
\begin{theorem}[$C^\infty$ antisymmetric gluing criterion]
\label{thm:c_inf_gluing}
Assume $\lambda(y)$ is $C^\infty$ and let $w(x,y)$ be a $C^\infty$ solution on a fundamental patch $[-x_+,x_+]\times I$ that is even in $x$. The antisymmetric extension defined successively by
\[
\widetilde w(x+2x_+,y)=-\widetilde w(x,y)
\]
is a globally $C^\infty$, $4x_+$-periodic solution if and only if
\[
\partial_x^{2k}w(x_+,y)=0
\qquad\text{for every }k\geq0\text{ and }y\in I.
\]
For a $C^r$ extension, it is enough and necessary to impose these conditions for $2k\leq r$.
\end{theorem}

\begin{proof}
At the seam $x=x_+$, the derivative from the adjacent antisymmetric patch is
\[
-\partial_x^n w(-x_+,y)=(-1)^{n+1}\partial_x^n w(x_+,y),
\]
where evenness was used in the equality. Odd-order derivatives therefore match automatically. For even $n$, matching is equivalent to $\partial_x^n w(x_+,y)=0$. These conditions are thus necessary and sufficient for all $x$-jets through the stated order to match. Differentiating the boundary identities in $y$ gives the corresponding mixed-jet matching. The equation is local and invariant under $w\mapsto-w$, and two antisymmetric extensions give period $4x_+$, which completes the proof.
\end{proof}

\begin{corollary}[Separable product family]
If $w(x,y)=g(y)\cos(kx)$ and $x_+=\pi/(2k)$, then every even $x$-derivative is proportional to $\cos(kx)$ and vanishes at the seam. Hence the antisymmetric extension is a smooth $4x_+$-periodic solution.
\end{corollary}

\subsection{Implications for the Goursat and Cauchy Problems}\label{subsec:implications-goursat-cauchy}
The geometric properties of this symmetric class of solutions directly motivate the boundary value problems analyzed in the remainder of this paper. 
In particular, the location of inflection lines and the role of symmetry provide the natural boundary framework for both Goursat and Cauchy formulations.

\medskip
\noindent {\bf The Goursat Problem.} The fixed inflection lines at $x=\pm\tfrac{\pi}{2}$ naturally define a fundamental solution patch, like the concave region where $w_{xx}<0$ on the physical domain $x \in [-\tfrac{\pi}{2}, \tfrac{\pi}{2}]$. In the hodograph plane, these inflection lines correspond to the boundaries of the domain for the function $x(s,t)$. This motivates the choice of Goursat boundary conditions used in the parametrix construction and in analyzing the singular problem at the origin:
\[
x_{\pm} = \pm \tfrac{\pi}{2}.
\]

\medskip
\noindent {\bf The Cauchy Problem.} The Cauchy data for the physical problem are typically imposed at a ``remote'' boundary, such as a line of constant $y=L$. Recalling the potential coordinate $u$ from \eqref{eq:defofu}, we note that the far field in physical space corresponds to the near field in hodograph space: as $y \to \infty$, one has $u \to 0^+$. A large but finite $y=L$ corresponds to a small positive value of the time-like variable, $c=u(L)$. Therefore, the physical Cauchy problem for $w$ at large $y$ becomes the initial value problem for $x$ at small $u$. For consistency with the global structure, we assume that any prescribed Cauchy data for $w$ and $w_y$ at $y=L$ also respect the symmetry conditions.

This asymmetry is the geometric origin of the \textit{``arrow of time''} for the mixed problem. If complete Cauchy data are instead prescribed on a large-$u$ terminal boundary while Goursat traces are prescribed independently on the characteristic axes, the Cauchy development already determines the traces that reach those axes. Generic independent choices are therefore overdetermined and fail the compatibility conditions for existence.

\section{The Well-Posed Cauchy--Goursat Problem}
\label{sec:cauchy-goursat}

Before analyzing the singular problem at the origin, we establish well-posedness on a truncated cone. This is a Cauchy--Goursat problem: a hyperbolic mixed-boundary problem with non-characteristic Cauchy data and characteristic Goursat data \cite{CourantHilbert1962}.

We derive a unified Volterra representation, formulate the corresponding hodograph weak solution, and prove existence, uniqueness, and boundary attainment. The underlying Volterra theorem is standard \cite{TricomibookIntegralEquations,Kress2014}; the paper-specific points are the weak boundary-data class, the improvement property of Lemma~\ref{lem:improvement}, and Proposition~\ref{prop:solution_attains_bcs}.

\subsection{The Finite-Cylinder Domain}
We define the domain of interest as follows:
\begin{equation}\label{DomainOmega}
\text{For a constant } c>0, \text{ let } \Omega = \{ (s,t) \in \mathbb{R}^2 \mid s>0, t>0, s+t>c \},
\end{equation}
which we refer to as the \emph{finite-cylinder domain}. The boundary $\partial \Omega$ of this domain's closure is composed of three segments:
\begin{itemize}
    \item A {\em non-characteristic} (Cauchy) segment: $\Gamma_C = \{ (s,t) \mid s+t=c,\\ s \ge 0,~ t \ge 0 \}$;
    \item Two {\em characteristic} (Goursat) segments: $\Gamma_1 = \{ (s,0) \mid s \ge c \}$ and $\Gamma_2 = \{ (0,t) \mid t \ge c \}$.
\end{itemize}

This setup showcases the interaction between the Cauchy line, where data are specified in a non-characteristic fashion, and the Goursat segments, which are characteristic boundaries. Together they form a truncated cone that avoids the singular corner, as shown in Figure \ref{fig:Setup}.

\begin{figure}[htbp]
\centering
\begin{tikzpicture}[scale=1]
    
    \def\cval{2.5}
    \def\sval{4}
    \def\tval{3.5}

    \begin{scope}[rotate=45]
        
        \fill[gray!20] (\sval,\tval) -- (0,\tval) -- (0,\cval) -- (\cval,0) -- (\sval,0) -- cycle;
        
        \draw[thick, blue] (0,\cval) -- (\cval,0); %
        \draw[thick, red] (\cval,0) -- (6,0);      %
        \draw[thick, red] (0,\cval) -- (0,6);      %

        \draw[dashed] (\sval,\tval) -- (\sval,0); %
        \draw[dashed] (\sval,\tval) -- (0,\tval); %

        \coordinate (testpoint) at (\sval,\tval);
        \coordinate (mid_cauchy) at (\cval/2, \cval/2);
        \coordinate (corner_s) at (\cval, 0);
        \coordinate (mid_sigma) at (\cval/2+3,0);
        \coordinate (corner_t) at (0, \cval);
        \coordinate (mid_tau) at (0,\cval/2+3);
        \coordinate (mid_s) at (\sval,\tval/2);
        \coordinate (mid_t) at (\sval/2,\tval);
        
    \node[above, left, xshift=38pt] at (mid_s) {$\sigma = s$};
    \node[above, xshift=-13pt] at (mid_t) {$\tau = t$};

    \end{scope}

    \fill (testpoint) circle (1.5pt) node[above] {$(s,t)$};
    
    \node[blue, below=5pt] at (mid_cauchy) {$\Gamma_C$: Cauchy Data ($\sigma+\tau=c$)};
    
    \node[red, below right=3pt, align=center] at (mid_sigma) {$\Gamma_1$ (on $\tau=0$)};
    \node[red, below left=2pt, align=center] at (mid_tau) {$\Gamma_2$ (on $\sigma=0$)};
    
\end{tikzpicture}
\caption{Geometry of the Cauchy--Goursat problem in rotated characteristic 
coordinates $(\sigma,\tau)$. The shaded region shows the domain of dependence 
for $(s,t)$ relevant for the integral representation and uniqueness proof; the truncated corner $\sigma+\tau \ge c > 0$ isolates the Cauchy data $\Gamma_C$ from the singular origin, ensuring well-posedness on finite cylinders.}\label{fig:Setup}
\end{figure}
We study the Cauchy--Goursat problem of finding $x$ such that
\begin{equation}
\label{prob:cauchy-goursat}
\left\{
\begin{aligned}
2\lambda(s+t)x_{st}+\lambda'(s+t)(x_s+x_t)&=0
&&\text{in }\Omega, \\
x(s,0)&=g(s)
&&\text{for }s\geq c \quad (\text{on }\Gamma_1), \\
x(0,t)&=h(t)
&&\text{for }t\geq c \quad (\text{on }\Gamma_2), \\
x(s,c-s)&=f(s)
&&\text{for }0\leq s\leq c \quad (\text{on }\Gamma_C), \\
(\partial_s+\partial_t)x(s,c-s)&=n(s)
&&\text{for }0\leq s\leq c \quad (\text{on }\Gamma_C).
\end{aligned}
\right.
\end{equation}
Here, $\partial_s + \partial_t$ is the derivative normal to the Cauchy segment $\Gamma_C$.

It is important to emphasize that not every choice of boundary data $\{g,h,f,n\}$ yields a meaningful solution. In particular, at the corners $(c,0)$ and $(0,c)$ the boundary conditions must align in a consistent way. To make this precise, we introduce compatibility notions, beginning with weak compatibility of the data.

\begin{definition}[Weak compatibility for the Cauchy--Goursat data]
\label{defn:weak-compatible}
 The Cauchy--Goursat boundary data $\{g,h,f,n\}$ are {\em weakly compatible} if
 \begin{enumerate}
    \item $\{g,h,f\}$ are $C^0$-compatible, meaning that there exists a $C^0$ function in a neighborhood of $\partial\Omega$ that agrees with the prescribed Dirichlet data on $\Gamma_1 \cup\Gamma_C \cup\Gamma_2$. This requires, at a minimum, that $g(c)=f(c)$ and $h(c)=f(0)$.
    \item $f$ is absolutely continuous, so that the tangential derivative is defined a.e.\ on $\Gamma_C$ and is integrable, i.e., $f' \in L^1(\Gamma_C)$.
    \item The prescribed normal derivative on $\Gamma_C$ satisfies $n \in L^1(\Gamma_C)$.
\end{enumerate}
\end{definition}

These assumptions guarantee continuity and integrability without imposing derivative matching at the two corners.

\noindent Having clarified the mildness of these requirements, we next specify what it means for a function to actually \emph{attain} such weakly compatible data on the boundary.

\begin{definition}[Attainment of boundary conditions]
\label{defn:attainment}
 A continuous function $x \in C^0(\bar{\Omega})$, where $\bar{\Omega}=\Omega \cup \partial\Omega$, attains the weakly compatible Cauchy--Goursat boundary conditions $\{g,h,f,n\}$ if $g,h,$ and $f$ are equal to $x$ restricted to the corresponding segments $\Gamma_1,\Gamma_2$ and $\Gamma_C$ in $\partial \Omega$ and, for almost all $s \in [0,c]$, the limit
 \[
 \lim_{z \to 0^+} \frac{x(s+z,c-s+z) -x(s,c-s)}{z}
 \]
 exists and is equal to $n(s)$.
\end{definition}

The natural function space for solutions to the governing PDE is dictated by its structure. For what follows, a \emph{classical solution} of the Cauchy--Goursat problem \eqref{prob:cauchy-goursat} is a function $x\in C^2(\Omega)\cap C(\bar\Omega)$ that satisfies the PDE in $\Omega$ and attains the boundary data in the sense of Definition~\ref{defn:attainment}. We additionally require that $x_s$ and $x_t$ admit locally integrable traces on the open boundary segments, with the one-sided contour limits used below converging in $L^1$. This trace condition allows the contour identity to be obtained first on offset contours and then passed to the boundary; it accommodates the integrable endpoint singularities that motivate weak compatibility. Such a classical solution necessarily has weakly compatible boundary data.

\begin{proposition}[Unified Cauchy--Goursat representation]\label{prop:unified-volterra}
Let $x$ be a classical solution of the Cauchy--Goursat problem \eqref{prob:cauchy-goursat}. Then, for every $(s,t)\in\Omega$,
\begin{align}
2\lambda(s+t)x(s,t) &=  \int_{(c-t)^+}^s \lambda'(\sigma+t)x(\sigma,t)\, d\sigma
+ \int_{(c-s)^+}^t \lambda'(s+\tau)x(s,\tau)\, d\tau \nonumber \\[0.5em]
&~\quad + \lambda(\max(c,t))x((c-t)^+,t)
+ \lambda(\max(s,c))x(s,(c-s)^+) \nonumber \\[0.5em]
&~~\quad + \int_{(c-t)^+}^{\min(c,s)} \lambda(c)(x_\sigma+x_\tau)|_{\sigma+\tau=c}\, d\sigma \nonumber \\[0.5em]
&~~~\quad + \int_{\min(c,s)}^s \lambda(\sigma+(c-s)^+) x_\sigma(\sigma, (c-s)^+)\, d\sigma \nonumber \\[0.5em]
&~~~~\quad + \int^t_{\min(c,t)} \lambda((c-t)^+ + \tau)x_\tau((c-t)^+,\tau)\, d\tau.
\label{eq:integral_solution}
\end{align}
\end{proposition}

\begin{proof}
Apply the calculation below first to a contour displaced a positive distance into $\Omega$. The assumed $L^1$ trace convergence permits passage to the boundary in each contour integral.
At a generic point $(\sigma,\tau)$, the PDE in \eqref{prob:cauchy-goursat} reads
\[
2\lambda(\sigma+\tau)x_{\sigma\tau}
+\lambda'(\sigma+\tau)(x_\sigma+x_\tau)=0.
\]
Thus, for
\[
\alpha=\lambda(\sigma+\tau)(x_\sigma\,d\sigma-x_\tau\,d\tau),
\]
we have
\[
d\alpha=-\bigl[2\lambda(\sigma+\tau)x_{\sigma\tau}
+\lambda'(\sigma+\tau)(x_\sigma+x_\tau)\bigr]\,d\sigma\wedge d\tau=0.
\]
Since $\Omega$ is convex, and hence simply connected, $\alpha$ is exact.
Using $\xi^+=\max(\xi,0)$, let $C$ be the generalized pentagonal contour with vertices
\[
\begin{array}{rl@{\hspace{0.5cm}}rl}
    V_1 &= ((c-t)^+, \min(c,t)), & V_2 &= (\min(c,s), (c-s)^+), \\[0.3em]
    V_3 &= (s, (c-s)^+),         & V_4 &= (s,t), \\[0.3em]
    V_5 &= ((c-t)^+, t).
\end{array}
\]
Figure~\ref{fig:CaseContours} shows the resulting contour in each of the four regimes determined by the position of $(s,t)$ relative to $c$.

\begin{figure}[htbp]
\centering
\begin{tabular}{cc}
    \begin{tikzpicture}[scale=0.9]
        \def\cval{2}
        \def\sval{4}
        \def\tval{3.5}

        \fill[blue!15] (\cval,-0.15) rectangle (\sval+0.5,0.15);
        \fill[blue!15] (-0.15,\cval) rectangle (0.15,\tval+0.5);
        \fill[blue!15] (\cval,-0.15) -- (0,\cval-0.15) -- (0,\cval+0.15) -- (\cval,0.15) -- cycle;

        \draw[blue, line width=2.5pt] (\cval,0) -- (\sval+0.5,0);
        \draw[blue, line width=2.5pt] (0,\cval) -- (0,\tval+0.5);
        \draw[blue, line width=2.5pt] (\cval,0) -- (0,\cval) node[midway, above, sloped] {$\sigma+\tau=c$};

        \draw[->] (-0.2,0) -- (5.5,0) node[right] {$\sigma$};
        \draw[->] (0,-0.2) -- (0,5) node[above] {$\tau$};

        \draw[red, line width=1.5pt, <-] (\sval,\tval) node[above right] {$P(s,t)$}
            -- (\sval,0) node[below] {$A$};
        \draw[red, line width=1.5pt, <-] (\sval,0) -- (\cval,0) node[below] {$B$};
        \draw[red, line width=1.5pt, <-] (\cval,0) -- (0,\cval) node[left] {$D$};
        \draw[red, line width=1.5pt, <-] (0,\cval) -- (0,\tval) node[left] {$E$};
        \draw[red, line width=1.5pt, <-] (0,\tval) -- (\sval,\tval);

        \fill (\sval,\tval) circle (2pt);
    \end{tikzpicture}
    &
    \begin{tikzpicture}[scale=0.9]
        \def\cval{3}
        \def\sval{1.5}
        \def\tval{4.5}

        \fill[blue!15] (\cval,-0.15) rectangle (5.2,0.15);
        \fill[blue!15] (-0.15,\cval) rectangle (0.15,5);
        \fill[blue!15] (\cval,-0.15) -- (0,\cval-0.15) -- (0,\cval+0.15) -- (\cval,0.15) -- cycle;

        \draw[blue, line width=2.5pt] (\cval,0) -- (5.2,0);
        \draw[blue, line width=2.5pt] (0,\cval) -- (0,5);
        \draw[blue, line width=2.5pt] (\cval,0) -- (0,\cval) node[midway, above, sloped] {$\sigma+\tau=c$};

        \draw[->] (-0.2,0) -- (5.5,0) node[right] {$\sigma$};
        \draw[->] (0,-0.2) -- (0,5) node[above] {$\tau$};

        \draw[red, line width=1.5pt, <-] (\sval,\tval) node[above right] {$P(s,t)$}
            -- (\sval,\cval-\sval) node[below] {$A$};
        \draw[red, line width=1.5pt, <-] (\sval,\cval-\sval) -- (0,\cval) node[left] {$B$};
        \draw[red, line width=1.5pt, <-] (0,\cval) -- (0,\tval) node[left] {$D$};
        \draw[red, line width=1.5pt, <-] (0,\tval) -- (\sval,\tval);

        \fill (\sval,\tval) circle (2pt);
    \end{tikzpicture} \\

    (a) Case I & (b) Case II \\[1em]

    \begin{tikzpicture}[scale=0.9]
        \def\cval{3}
        \def\sval{4.5}
        \def\tval{1.5}

        \fill[blue!15] (\cval,-0.15) rectangle (5.2,0.15);
        \fill[blue!15] (-0.15,\cval) rectangle (0.15,5);
        \fill[blue!15] (\cval,-0.15) -- (0,\cval-0.15) -- (0,\cval+0.15) -- (\cval,0.15) -- cycle;

        \draw[blue, line width=2.5pt] (\cval,0) -- (5.2,0);
        \draw[blue, line width=2.5pt] (0,\cval) -- (0,5);
        \draw[blue, line width=2.5pt] (\cval,0) -- (0,\cval) node[midway, above, sloped] {$\sigma+\tau=c$};

        \draw[->] (-0.2,0) -- (5.5,0) node[right] {$\sigma$};
        \draw[->] (0,-0.2) -- (0,5) node[above] {$\tau$};

        \draw[red, line width=1.5pt, <-] (\sval,\tval) node[above right] {$P(s,t)$}
            -- (\sval,0) node[below] {$A$};
        \draw[red, line width=1.5pt, <-] (\sval,0) -- (\cval,0) node[below] {$B$};
        \draw[red, line width=1.5pt, <-] (\cval,0) -- (\cval-\tval,\tval) node[left] {$D$};
        \draw[red, line width=1.5pt, <-] (\cval-\tval,\tval) -- (\sval,\tval);

        \fill (\sval,\tval) circle (2pt);
    \end{tikzpicture}
    &
    \begin{tikzpicture}[scale=0.9]
        \def\cval{4}
        \def\sval{1.5}
        \def\tval{3.5}

        \fill[blue!15] (\cval,-0.15) rectangle (5.2,0.15);
        \fill[blue!15] (-0.15,\cval) rectangle (0.15,5);
        \fill[blue!15] (\cval,-0.15) -- (0,\cval-0.15) -- (0,\cval+0.15) -- (\cval,0.15) -- cycle;

        \draw[blue, line width=2.5pt] (\cval,0) -- (5.2,0);
        \draw[blue, line width=2.5pt] (0,\cval) -- (0,5);
        \draw[blue, line width=2.5pt] (\cval,0) -- (0,\cval) node[midway, above, sloped] {$\sigma+\tau=c$};

        \draw[->] (-0.2,0) -- (5.5,0) node[right] {$\sigma$};
        \draw[->] (0,-0.2) -- (0,5) node[above] {$\tau$};

        \draw[red, line width=1.5pt, <-] (\sval,\tval) node[above right] {$P(s,t)$}
            -- (\sval,\cval-\sval) node[left] {$A$};
        \draw[red, line width=1.5pt, <-] (\sval,\cval-\sval) -- (\cval-\tval,\tval) node[below left] {$B$};
        \draw[red, line width=1.5pt, <-] (\cval-\tval,\tval) -- (\sval,\tval);

        \fill (\sval,\tval) circle (2pt);
    \end{tikzpicture} \\

    (c) Case III & (d) Case IV
\end{tabular}
\caption{The four geometric configurations of the pentagonal contour $C$ in the proof of Proposition~\ref{prop:unified-volterra}, according to the position of $(s,t)$ relative to $c$: (a) $s>c, t>c$, (b) $0 \le s \le c, t>c$, (c) $s>c, 0 \le t \le c$, and (d) $0 \le s \le c, 0 \le t \le c, s+t>c$. The prescribed boundary with known Cauchy--Goursat data is highlighted in blue.} \label{fig:CaseContours}
\end{figure}

Exactness of $\alpha$ and separation of the characteristic from boundary paths give

\begin{equation}\label{integralequationK5}
\int_{(c-t)^+}^s \lambda(\sigma+t) x_\sigma(\sigma,t) \, d\sigma + \int_{(c-s)^+}^t \lambda(s+\tau) x_\tau(s,\tau) \, d\tau = \mathcal{J}_{5},
\end{equation}

Here $\mathcal{J}_{5}$ is the contribution from $V_1V_2$, $V_2V_3$, and $V_5V_1$; those edges lie on the prescribed boundary or have zero length.

Integrating by parts on $V_3V_4$ and $V_4V_5$ transfers derivatives from $x$ to $\lambda$; collecting the boundary contributions gives \eqref{eq:integral_solution}.
\end{proof}

The terms depending only on the boundary data combine into $\mathcal{G}(s,t)$:
\begin{align}
\mathcal{G}(s,t) &=  \begin{cases} 
\lambda(c) f(s) & s \leq c, \\[0.5em]  
2 \lambda(s) g(s) - \lambda(c) g(c) - \int_c^s g(\sigma) \lambda'(\sigma)\, d\sigma & s > c  
\end{cases} \nonumber \\[0.5em]
&~ \quad + \begin{cases} 
\lambda(c) f(c-t) & t \leq c, \\[0.5em]  
2 \lambda(t) h(t) - \lambda(c) h(c) - \int_c^t h(\tau) \lambda'(\tau)\, d\tau & t > c  
\end{cases} \nonumber \\[0.5em]
&~~\quad + \int_{(c-t)^+}^{\min(c,s)} \lambda(c) n(\sigma)\, d\sigma.  
\label{eq:boundary_forcing}
\end{align}

Substituting the prescribed boundary data into \eqref{eq:integral_solution} gives the equivalent Volterra relation
\begin{equation}
\label{eq:T_simplified}
\begin{aligned}
x(s,t) = \frac{1}{2\lambda(s+t)} \Big( \mathcal{G}(s,t)
&+ \int_{(c-t)^+}^s \lambda'(\sigma+t)x(\sigma,t)\, d\sigma \\[0.5em]
&+ \int_{(c-s)^+}^t \lambda'(s+\tau)x(s,\tau)\, d\tau \Big).
\end{aligned}
\end{equation}
For $u\in C(\bar\Omega)$, let $(\mathcal Au)(s,t)$ denote the right-hand side of \eqref{eq:T_simplified} with $x$ replaced by $u$. Thus every classical solution satisfies the fixed-point equation $x=\mathcal Ax$.

\begin{remark}
\label{rem:Ck-consistency}
    Weak compatibility implies $\mathcal{G}\in C^0(\Omega)$, including across $s=c$ and $t=c$. In general $\mathcal{G}$ is not differentiable, even for $C^k$-compatible data. For data smooth enough for the displayed one-sided derivatives to exist, one finds, for $t>c$,
    \[
    \mathcal{G}_s(c^-,t) = \lambda(c)\big(f'(c)+n(c)\big) \text{  and  ~}  
    \mathcal{G}_s(c^+,t) = 2\lambda(c) g'(c) - g(c)\lambda'(c).
    \]
    Since $\lambda(c)$ and $\lambda'(c)$ are independent of the data, equality would require the additional conditions $f'(c)+n(c)-2g'(c)=0$ and $g(c)=0$, which are not consequences of $C^k$ compatibility.
\end{remark}

\subsection{Existence and Uniqueness of Hodograph Weak Solutions}
The representation \eqref{eq:integral_solution} contains boundary derivatives and therefore applies only to classical solutions. In contrast, \eqref{eq:T_simplified} involves only values of $x$ and is meaningful on $\bar{\Omega}$ for continuous functions.

\begin{definition}[Hodograph weak solution]\label{defn:hodograph-weak}
A continuous function $x:\bar{\Omega} \to \mathbb{R}$ is called a \emph{hodograph weak solution}
to the Cauchy--Goursat problem~\eqref{prob:cauchy-goursat} if it satisfies the integral relation~\eqref{eq:T_simplified} at every point of $\bar{\Omega}$.
\end{definition}

\begin{remark}[Contrast to other types of solutions]\label{rem:weak_sol_context}
This transformed hyperbolic notion is distinct from the solution concepts commonly used for elliptic Monge--Amp\`ere equations. An Alexandrov solution is a convex function whose Monge--Amp\`ere measure, defined through the image of its subdifferential, equals the prescribed measure \cite{Alexandrov}. A viscosity solution, by contrast, is characterized by comparison with smooth test functions that touch the candidate solution from above or below \cite{CrandallIshiiLions1992}. Here we instead require a continuous function in the transformed $(p,y)$ plane to satisfy the Volterra relation \eqref{eq:T_simplified}.

Definition~\ref{defn:hodograph-weak} is a solution concept for the transformed Cauchy--Goursat problem. It does not, by itself, solve the original Monge--Amp\`ere equation. If a hodograph solution has sufficient regularity for the differential-form reconstruction and, for each fixed $u$, the map $p\mapsto x(u,p)$ is injective with $x_p$ nonzero, then it is a candidate for reconstruction on a strictly partially convex physical patch. This is not a mere technicality; failure of this monotonicity produces a fold and prevents the inverse hodograph map from defining a single physical graph as shown below in Example~\ref{ex:fold-singularity}. 
\end{remark}

On a bounded rectangle $R=[0,S]\times[0,T]$, Eq.~\eqref{eq:T_simplified} is a two-dimensional Volterra equation of the second kind \cite{TricomibookIntegralEquations}. The following standard theorem applies.

\begin{theorem}
\label{thm:existence-uniqueness}
Let $\Omega_{S,T} := \bar\Omega \cap \big([0,S]\times[0,T]\big)$, which is compact since $\bar\Omega$ is closed and $[0,S]\times[0,T]$ is bounded. Suppose the lower-limit functions $a$ and $b$ are continuous and Volterra-domain invariant: whenever $(s,t)\in\Omega_{S,T}$, the sampled points $(\sigma,t)$ for $a(t)\leq\sigma\leq s$ and $(s,\tau)$ for $b(s)\leq\tau\leq t$ also lie in $\Omega_{S,T}$. Let $x$ satisfy the two-dimensional Volterra integral equation of the second kind
\[
x(s,t) = \mathcal{F}(s,t)
+ \int_{a(t)}^{s} K_1(s,t,\sigma)\,x(\sigma,t)\,d\sigma
+ \int_{b(s)}^{t} K_2(s,t,\tau)\,x(s,\tau)\,d\tau,
\]
for $(s,t)\in \Omega_{S,T}$.
If the forcing term $\mathcal{F}$ is continuous on $\Omega_{S,T}$ and the kernels are continuous on their compact Volterra domains
\[
\begin{aligned}
\mathcal D_1&:=\{(s,t,\sigma):(s,t)\in\Omega_{S,T},\ a(t)\leq\sigma\leq s\},\\
\mathcal D_2&:=\{(s,t,\tau):(s,t)\in\Omega_{S,T},\ b(s)\leq\tau\leq t\},
\end{aligned}
\]
then there exists a unique continuous solution $x$ on $\Omega_{S,T}$.
\end{theorem}

\begin{equation}
\label{eq:K_op}
(\mathcal{K}x)(s,t):=
\int_{a(t)}^{s}K_1(s,t,\sigma)x(\sigma,t)\,d\sigma+
\int_{b(s)}^{t}K_2(s,t,\tau)x(s,\tau)\,d\tau .
\end{equation}

\begin{remark}
For our setup, take \(a(t)=(c-t)^+\), \(b(s)=(c-s)^+\), and
\begin{equation}\label{DefFK1K2}
\mathcal{F}(s,t) = \frac{\mathcal{G}(s,t)}{2\lambda(s+t)}, 
~~K_1(s,t,\sigma) = \frac{\lambda'(\sigma+t)}{2\lambda(s+t)}, 
~~ K_2(s,t,\tau) = \frac{\lambda'(s+\tau)}{2\lambda(s+t)},
\end{equation}
where $\mathcal{G}$ is defined in \eqref{eq:boundary_forcing}.
The truncation $s+t\ge c>0$ keeps these kernels bounded; this is precisely why the finite-cylinder problem is a regular precursor to the singular problem.
\end{remark}

\begin{proof}[Proof sketch of Theorem~\ref{thm:existence-uniqueness}]
Write $\mathcal{A}x=\mathcal{F}+\mathcal{K}x$, where $\mathcal{K}$ is the integral operator in the statement, on $C(\Omega_{S,T})$ with the sup norm. If $M=\max\{\sup|K_1|,\sup|K_2|\}$ and $x_{n+1}=\mathcal{A}x_n$, the standard Volterra induction gives
\[
\|x_{n+1}-x_n\|_\infty\leq L\frac{(2M(S+T))^n}{n!},
\qquad L=\|x_1-x_0\|_\infty .
\]
The $M$-test gives a uniformly convergent fixed-point sequence. Applying the same estimate to the difference of two fixed points proves uniqueness.
\end{proof}

For the present data, the choices in \eqref{DefFK1K2} recover the operator $\mathcal A$ defined above; equivalently, \eqref{eq:T_simplified} is the corresponding two-dimensional Volterra equation.

We also define the normal derivative operator along the Cauchy segment $\Gamma_C$ by
\begin{equation}
\label{eq:Dn_def}
D_n x(\xi) := \lim_{z \to 0^+} \frac{x(s_z, t_z) - x(\xi, c-\xi)}{z}, 
~~(s_z, t_z) = (\xi+z, c-\xi+z),
\end{equation}
for a.e.\ $\xi \in (0,c)$.

Our goal now is to highlight a key property of the operator $\mathcal{A}$. 
The next lemma shows that $\mathcal{A}$ acts as an ``improvement'' operator for boundary conditions: 
starting from an input $x_0$ that satisfies only the Dirichlet data on $\partial \Omega$, 
the output $x_1 = \mathcal{A}x_0$ automatically also satisfies the normal derivative condition 
along the Cauchy segment $\Gamma_C$. This shows that enforcement of the derivative 
condition is built into the structure of $\mathcal{A}$, rather than requiring successive iterations.

\begin{lemma}
\label{lem:improvement}
Let $x_0 \in C(\overline{\Omega})$ be a continuous function that satisfies the Dirichlet boundary conditions $\{g,h,f\}$ on $\partial \Omega$. Then the function $x_1 = \mathcal{A}x_0$, defined by \eqref{eq:T_simplified}, not only agrees with $\{g,h,f\}$ on $\partial \Omega$ but also satisfies the normal derivative condition
\[
D_n x_1(\xi) = n(\xi) \quad \text{for almost  everywhere\ } \xi \in \Gamma_C.
\]
\end{lemma}

\begin{proof}
By hypothesis $x_0$ satisfies $\{g,h,f\}$ on $\partial\Omega$; that $x_1=\mathcal A x_0$ does too follows directly from the definition of $\mathcal A$. It remains to verify the normal derivative condition on $\Gamma_C$.

Fix $\xi \in (0,c)$; by \eqref{eq:Dn_def},
\[
D_n x_1(\xi) = \lim_{z \to 0^+} \frac{x_1(s_z,t_z) - x_1(\xi, c-\xi)}{z},
\]
and we show this equals $n(\xi)$ for a.e.\ $\xi \in (0,c)$.

From the definition of the operator $\mathcal{A}$, and noting that $0 < s_z, t_z < c$, we obtain
\[
x_1(s_z, t_z) = \frac{1}{2\lambda(c+2z)} \Big( \mathcal{G}(s_z, t_z) + I_1(z) + I_2(z) \Big),
\]
where we have set
\[
I_1(z) = \int_{\xi-z}^{\xi+z} \lambda'(\sigma+t_z)\, x_0(\sigma,t_z)\, d\sigma,
\quad
I_2(z) = \int_{c-\xi-z}^{c-\xi+z} \lambda'(s_z+\tau)\, x_0(s_z,\tau)\, d\tau.
\]

Since the integrands of $I_1$ and $I_2$ are continuous in $z$, it follows that $I_1(0) = I_2(0) = 0$.
Consequently, substituting into the definition of the normal derivative operator gives
\[
D_n x_1(\xi)
= \frac{\lambda(c) D_n[\mathcal{G}(s_z,t_z)] - 2 \lambda'(c)\mathcal{G}(\xi,c-\xi)}{2\lambda(c)^2}
+ \frac{1}{2\lambda(c)} \left. \frac{\partial}{\partial z}\big( I_1(z) + I_2(z) \big)\right|_{z=0}.
\]

Also, by continuity of the integrands, the last term simplifies to
\[
\left. \frac{\partial}{\partial z}\big( I_1(z) + I_2(z) \big) \right|_{z=0}
= 4 \lambda'(c) x_0(\xi, c-\xi)
= 4 \lambda'(c) f(\xi).
\]

For the term involving $\mathcal{G}$, note from \eqref{eq:boundary_forcing} and the fact that
$s_z, t_z < c$ that
\begin{align*}
\mathcal{G}(s_z,t_z)
= & \lambda(c)\Big( f(s_z) + f(c-t_z) + \int_{c-t_z}^{s_z} n(\sigma)\, d\sigma \Big)\\[0.5em]
= & \lambda(c)\Big( f(\xi+z) + f(\xi-z) + \int_{\xi-z}^{\xi+z} n(\sigma)\, d\sigma \Big).
\end{align*}

By the Lebesgue differentiation theorem, and since $n \in L^1(\Gamma_C)$, the following holds for almost every \ $\xi \in (0,c)$.
\[
D_n[\mathcal{G}(\xi)]
= \lim_{z \to 0^+} \frac{\lambda(c)\big(f(\xi+z)+f(\xi-z)-2f(\xi)\big)}{z}
+ 2\lambda(c) n(\xi).
\]

Combining this with the previous computation, we obtain
\begin{align*}
D_n x_1(\xi)
= \frac{1}{2\lambda(c)} \Bigg(
    &\lim_{z \to 0^+} \frac{\lambda(c)\big(f(\xi+z)+f(\xi-z)-2f(\xi)\big)}{z}\\[0.5em]
    &~~+ 2\lambda(c)n(\xi) - 4\lambda'(c)f(\xi) + 4\lambda'(c)f(\xi)
\Bigg).
\end{align*}

The last two terms cancel, leaving
\[
D_n x_1(\xi)
= \lim_{z \to 0^+} \frac{f(\xi+z)+f(\xi-z)-2f(\xi)}{2z} + n(\xi),
\]
which holds for almost every $\xi \in (0,c)$. Since $f$ is absolutely continuous, $f'(\xi)$ exists for a.e.\ $\xi$, and at every such point the symmetric difference quotient above is the average of the one-sided limits $\pm f'(\xi)$, hence $0$, leaving $D_n x_1(\xi) = n(\xi)$.
\end{proof}

The improvement property of $\mathcal{A}$ developed in Lemma~\ref{lem:improvement} plays a central role in the analysis. It ensures that the boundary conditions are precisely encoded in the operator framework. We now turn to the main proposition of this section, which establishes both the existence of a unique solution and its attainment of the full Cauchy--Goursat boundary conditions.

\begin{proposition}
\label{prop:solution_attains_bcs}
Let the Cauchy--Goursat data $\{g,h,f,n\}$ be weakly compatible. 
Then the associated Volterra integral equation
\[
x(s,t) = \mathcal{F}(s,t) 
+ \int_{(c-t)^+}^s K_1(s,t,\sigma)\,x(\sigma,t)\,d\sigma 
+ \int_{(c-s)^+}^t K_2(s,t,\tau)\,x(s,\tau)\,d\tau 
\]
possesses a unique continuous solution $x \in C(\overline{\Omega})$. 
Moreover, the solution attains the prescribed Cauchy--Goursat boundary conditions $\{g,h,f,n\}$.
\end{proposition}

\begin{proof}
The proof proceeds in two main steps: establishing the existence of a unique solution and demonstrating that this solution attains the boundary conditions.

\medskip
\noindent \textbf{Existence and uniqueness.} 
Given weakly compatible data, the forcing term which is defined by
\[
\mathcal{F}(s,t) = \frac{\mathcal{G}(s,t)}{2\lambda(s+t)},
\]
is continuous on any compact subdomain of the form
\[
\bar{\Omega} \cap [0,S]\times[0,T] 
= \{(s,t)\mid 0\leq s \leq S,\; 0\leq t \leq T,\; s+t \geq c >0\}.
\]
The kernels \(K_1\) and \(K_2\) are likewise continuous on their respective domains.  

By Theorem~\ref{thm:existence-uniqueness}, there exists a unique continuous solution \(x_{S,T}\) on each rectangle \([0,S]\times[0,T]\) intersected with \(\bar{\Omega}\). 
Exhausting \(\bar{\Omega}\) by such rectangles establishes the existence of a unique continuous solution \(x \in C(\bar{\Omega})\). 
This solution is the fixed point of the operator \(\mathcal{A}\), i.e., \(x = \mathcal{A}x\).  

\medskip

\noindent \textbf{Attainment of boundary conditions.} 
To show that the unique solution $x$ satisfies the prescribed boundary conditions, we examine the sequence of successive approximations used to construct it.

We begin by constructing a suitable initial guess $x_0$. By weak compatibility, there exists a continuous boundary function on $\partial \Omega$ that agrees with the Dirichlet data $\{g,h,f\}$. By the Tietze extension theorem, this extends to a continuous function $x_0 \in C(\bar{\Omega})$ agreeing with $\{g,h,f\}$ on $\partial\Omega$.

Next, define the sequence $x_k = \mathcal{A}x_{k-1}$ for $k \geq 1$. By Theorem~\ref{thm:existence-uniqueness}, $x_k \to x$ uniformly on compact subdomains. Moreover, since $x_0$ agrees with $\{g,h,f\}$ on $\partial\Omega$, Lemma~\ref{lem:improvement} and induction guarantee that every iterate $x_k$ also satisfies the Dirichlet conditions.

Passing to the limit, $x$ inherits $\{g,h,f\}$ on $\partial\Omega$, since every $x_k$ does.

Finally, since the unique solution $x$ satisfies the Dirichlet conditions, Lemma~\ref{lem:improvement} applies to $x$ itself. The lemma asserts that $\mathcal{A}x$ satisfies the normal derivative condition. But $x$ is the unique fixed point of $\mathcal{A}$, so $\mathcal{A}x = x$. Hence, 
\[
D_n x(\xi) = n(\xi) \quad \text{for a.e. } \xi \in [0,c].
\]

This completes the proof that the unique solution to the Volterra equation attains all prescribed Cauchy--Goursat boundary conditions.
\end{proof}

\begin{corollary}[Uniqueness of classical solutions]\label{cor:classical-uniqueness}
Let $x_1, x_2 \in C^2(\Omega)$ be classical solutions of the Cauchy--Goursat problem with the same boundary data $\{g,h,f,n\}$. Then $x_1 = x_2$ on $\bar\Omega$.
\end{corollary}

\begin{proof}
Each $x_i$ is continuous on $\bar\Omega$ and, by Proposition~\ref{prop:unified-volterra}, satisfies \eqref{eq:integral_solution}, equivalently \eqref{eq:T_simplified}, on $\bar\Omega$. Hence each $x_i$ is a hodograph weak solution in the sense of Definition~\ref{defn:hodograph-weak}. By Theorem~\ref{thm:existence-uniqueness}, applied on $\Omega_{S,T}$ for every $S,T$, $x_1=x_2$ on each $\Omega_{S,T}$ and hence on $\bar\Omega=\bigcup_{S,T}\Omega_{S,T}$.
\end{proof}
 
\section{Singularities and Their Resolution} \label{sec:singularities}

A central theme in the theory of PDEs is that solution regularity depends not only on the smoothness of the data but also on the geometry of the problem. Our analysis in the hodograph plane faces two distinct challenges, which this section addresses in turn. The first is the existence of a well-posed solution $x(s,t)$, which is complicated by the \emph{corner singularities} inherent to the Cauchy--Goursat geometry. They arise at the transition from non-characteristic (Cauchy) to characteristic (Goursat) boundaries and are present even for solutions that are smooth in the interior. The second challenge is the physical interpretation of a solution once it is found. The inverse map from the hodograph plane back to physical space can fail, creating \emph{fold singularities} \cite{brander2017pseudospherical} where the solution $x(s,t)$ no longer represents a valid, non-degenerate surface $w(x,y)$.

Our approach is to first build a framework to resolve the corner singularity and construct a hodograph solution, and then to investigate the additional condition needed for that solution to be free of folds. We begin with an exact ``arcsine" solution showing how singular boundary data arise naturally. A formal local balance then identifies a square-root cusp profile near the characteristic boundaries. This motivates a \emph{parametrix} that models the singular structure and reduces the problem to finding a more regular \emph{corrector}. Finally, we analyze the inverse problem and the formation of fold singularities.

\begin{example} [The ``arcsine" solution]\label{ex:arcsin}
Recall from \S\ref{subsec:geometric_setting} the running example product solution $w_0(x,y) = e^{-y}\cos x$, for which $\lambda(y) = e^{-y}$ and, on the fundamental patch $|x|<\pi/2$, the inverse slope map is $x = \arcsin(-p/u)$ with $u = e^{-y}$ and $p = (w_0)_x = -e^{-y}\sin x$ (Eq.~\eqref{eq:running_product_derivatives}). Writing this in the characteristic coordinates $s = \tfrac12(u+p)$, $t = \tfrac12(u-p)$ of \S\ref{sec:hodograph} (so $u=s+t$, and hence $\lambda(u)=u$) gives
\begin{equation}\label{eq:x_a}
x_a(s,t) = \arcsin\left(-\frac{p}{u}\right) = \arcsin\left(\frac{t-s}{t+s}\right),
\end{equation}
and a direct calculation verifies that $x_a$ solves the governing equation
\begin{equation}
    2(s+t) x_{st} + x_s + x_t = 0,
    \label{eq:model_pde}
\end{equation}
for $s>0,~t>0$. This solution is analytic in the open first quadrant.

Let us examine the boundary data it induces on the following domain $\Omega = \{ (s,t) \in \mathbb{R}^2 \mid s>0, t>0, s+t>c \}$, for some $c>0$.
\begin{itemize}
    \item \textbf{Goursat Data on $\Gamma_1$ and $\Gamma_2$:}
    \begin{align*}
        \text{For } s \ge c, &\text{ the boundary value on } \Gamma_1 \text{ is } g(s) = x_a(s,0) = \arcsin(-1) = -\tfrac{\pi}{2}.\\[0.5em]
        \text{For } t \ge c, &\text{ the boundary value on } \Gamma_2 \text{ is } h(t) = x_a(0,t) = \arcsin(1) = \tfrac{\pi}{2}.
    \end{align*}

    \item \textbf{Cauchy Data on $\Gamma_C$:}
    For $0 \le s \le c$, we have $t=c-s$. The Dirichlet data is
    $$
    f(s) = x_a(s, c-s) = \arcsin\left(\frac{c-s-s}{c-s+s}\right) = \arcsin\left(\frac{c-2s}{c}\right).
    $$
    
    The associated normal derivative on $\Gamma_C$ is then computed as: 
    $$
    n(s) = (\partial_s + \partial_t)x_a(s,t)\big|_{t=c-s} = \frac{2s-c}{c\sqrt{s(c-s)}}.
    $$
\end{itemize}
\end{example}

This example illustrates that the boundary data set $\{g,h,f,n\}$ generated by the smooth interior solution $x_a$ is appropriate for our weak compatibility of boundary data, as seen in Definition~\ref{defn:weak-compatible}. Namely,
 \begin{enumerate}
    \item \textbf{Continuity of Dirichlet data.} The boundary data are continuous at the corners:
    \begin{align*}
(c,0): ~ & g(c) = -\tfrac{\pi}{2}, 
& f(c) &= \arcsin\!\left(\tfrac{c-2c}{c}\right) = -\tfrac{\pi}{2}, 
& \text{so } g(c) = f(c). \\[0.7em]
(0,c): ~ & h(c) = \tfrac{\pi}{2}, 
& f(0) &= \arcsin\!\left(\tfrac{c}{c}\right) = \tfrac{\pi}{2}, 
& \text{so } h(c) = f(0).
\end{align*}
    \item \textbf{Absolute continuity of $f$.} The tangential derivative is given by $f'(s) = -1/\sqrt{s(c-s)}$. While this derivative is singular at $s=0$ and $s=c$, it is integrable over $[0,c]$. Thus, $f$ is absolutely continuous and $f' \in L^1(\Gamma_C)$.
    \item \textbf{Integrability of $n$.} The normal derivative $n(s)$ has inverse-square-root singularities at the endpoints. These singularities are integrable, and thus $n \in L^1(\Gamma_C)$.
\end{enumerate}

The crucial point is that, although the data satisfy all the weak compatibility requirements, the normal derivative $n(s)$ is manifestly unbounded as $s \to 0^+$ and $s \to c^-$. Hence this data set cannot be the trace of any $C^1(\bar{\Omega})$ function. This example illustrates the necessity of weak compatibility, as smooth interior solutions may still generate singular boundary traces at the corners.

\subsection{Asymptotic Analysis near a Goursat Boundary}\label{sec:corner-singularity}\label{sec:asymptotics}
To better understand the solution’s structure near a characteristic (Goursat) boundary, 
we carry out a local analysis of the behavior as $s \to 0^+$ with $t>c$. The same reasoning applies in both the finite and infinite cylinder cases. Motivated by the arcsine solution \eqref{eq:x_a}, which exhibits a square-root singularity in its gradient, 
we posit the ansatz
\begin{align*}
x(s,t) = h(t) + A(t)\sqrt{s} + o(\sqrt{s}), \quad s \to 0^+,
\end{align*}
where $x(0,t) = h(t)$ is the prescribed Goursat data on that boundary.

The relevant partial derivatives become:
\begin{align*}
x_s(s,t)  &= \tfrac{1}{2}A(t)\,s^{-1/2} + o\left(s^{-1/2}\right), \\[0.5em]
x_t(s,t)  &= h'(t) + A'(t)\sqrt{s} + o\left(\sqrt{s}\right), \\[0.5em]
x_{st}(s,t) &= \tfrac{1}{2}A'(t)\,s^{-1/2} + o\left(s^{-1/2}\right).
\end{align*}

Since $\lambda,\lambda' \in C^0$, as $s\to0^+$ we have $\lambda(s+t)\to\lambda(t)$ and $\lambda'(s+t)\to\lambda'(t)$:
$$
\lambda(s+t) = \lambda(t) + o(1), \quad
\lambda'(s+t) = \lambda'(t) + o(1).
$$
Substituting these relations into the governing partial differential equation $2\lambda(s+t) x_{st} + \lambda'(s+t) (x_s + x_t) = 0$,
 we obtain that the following expression must vanish:
$$
2(\lambda(t) + o(1))\left(\tfrac{1}{2}A'(t)s^{-1/2} + \dots\right)
+ (\lambda'(t) + o(1))\left(\tfrac{1}{2}A(t)s^{-1/2} + h'(t) + \dots\right).
$$

To find the condition for this equation to hold, we collect the coefficients of the most singular term, $s^{-1/2}$:
\[
2\lambda(t)\left(\tfrac{1}{2}A'(t)\right) + \lambda'(t)\left(\tfrac{1}{2}A(t)\right) = 0.
\]
Balancing the leading-order terms gives a first-order ordinary differential equation for the amplitude function $A(t)$:
\[
\lambda(t)A'(t) + \tfrac{1}{2}\lambda'(t)A(t) = 0
\]
and the solution is given by
\begin{equation}
A(t) = M[\lambda(t)]^{-1/2},
\label{eq:local_singularity}
\end{equation}
where $M$ is a constant of integration.

\subsection{The Parametrix and Corrector Solutions}\label{sec:parametrix}
The previous calculation shows that, for a solution to exhibit the square-root cusp behavior near the boundary $s=0$, the amplitude $A(t)$ cannot be arbitrary. It has to scale like $\lambda(t)^{-1/2}$. This motivates a natural deformation of the arcsine solution from Example~\ref{ex:arcsin}, namely
\begin{equation}\label{eq:x_p}
x_p(s,t) = \arcsin\!\left(\frac{\lambda(t)-\lambda(s)}{\lambda(t)+\lambda(s)}\right).
\end{equation}
Expanding this expression for small $s$ yields
\[
\arcsin\!\left(\frac{\lambda(t)-\lambda(s)}{\lambda(t)+\lambda(s)}\right)
\approx
\arcsin\!\left(1 - \frac{2\lambda(s)}{\lambda(t)}\right)
\approx \frac{\pi}{2} - 2\sqrt{\tfrac{\lambda(s)}{\lambda(t)}}
\approx \frac{\pi}{2} - \tfrac{M}{\sqrt{\lambda(t)}}\sqrt{s},
\]
for some constant $M$, consistent with the square-root asymptotics obtained in \eqref{eq:local_singularity}.

We think of $x_p$ as a “parametrix,” i.e., an approximate solution of the PDE that satisfies the Goursat boundary conditions $x(s,0)=-\pi/2$ and $x(0,t)=\pi/2$, which are discontinuous at the origin, while also modeling the leading singular behavior of $x$. While the term parametrix typically denotes an approximate inverse operator \cite{HormanderbookI}, our usage follows its strategic purpose of providing an approximate solution that captures the principal singular behavior, thereby regularizing the problem for the remainder of the solution, the {\em corrector term}.

The governing PDE for $x$ is given by $L[x]=0$, where the linear operator is defined in \eqref{eq:LinearOperator}. Let us now define the residual by
\[
\mathcal{H}(s,t) := L[x_p],
\]
which captures the result of applying the PDE operator to the parametrix. In general, this residual is not zero. We therefore introduce the \emph{corrector} solution $x_{\mathrm{corr}}(s,t)$, chosen such that
\[
x_{\mathrm{base}} = x_p + x_{\mathrm{corr}}
\]
satisfies both the PDE
$L[x_{\mathrm{base}}] = 0$
and the appropriate boundary conditions.  

Since $x_p$ was constructed to satisfy the original Goursat data, the corrector $x_{\mathrm{corr}}$ necessarily satisfies homogeneous Goursat boundary conditions. Thus, the problem for $x_{corr}$ reduces to solving a Goursat problem with homogeneous boundary data and forcing term
$-L[x_p] = -\mathcal{H}.$ This strategy is effective only if the residual $\mathcal{H}$ is smoother than the singular parametrix $x_p$.

For the analysis that follows, we fix the domains for $S>0,~T>0$ as
\begin{equation} \label{domainforresidual}
    R := [0,S]\times[0,T], \quad R_0 := R \setminus \{\mathbf{0}\}.
\end{equation}

We first establish the regularity of the residual $\mathcal{H}$, which will play a central role in the construction of the corrector.

\begin{lemma}[Regularity of the Residual]\label{lem:regularity-residual}
Suppose
\[
\lambda\in C^3([0,S+T])\cap C^\infty((0,S+T]),
\qquad \lambda(0)=0,\qquad \lambda'(0)>0,
\]
and $\lambda(u)>0$ for $u>0$. Then the residual
$\mathcal{H}(s,t) = L[x_p]$
is smooth on the domain $(0,S)\times(0,T)$. Moreover, there exists a constant $M_H>0$ such that, for all $(s,t)\in R_0$, we have the bound 
\[
\left|\mathcal{H}(s,t)\right| \leq M_H \,\frac{\sqrt{st}}{s+t}.
\]
\end{lemma}

\begin{proof}
A direct computation of $\mathcal{H}(s,t) = L[x_p]$ gives
\begin{align*}
\mathcal{H}(s,t) = & -\frac{1}{\sqrt{\lambda(s)\lambda(t)}\,(\lambda(s)+\lambda(t))^2} \bigg[-\lambda(s)^2 \lambda'(t)\lambda'(s+t) \\[0.5em]
&~+ \lambda(t)\lambda'(s)\big(-\lambda(s+t)\lambda'(t) + \lambda(t)\lambda'(s+t)\big) \\[0.5em]
&\quad+ \lambda(s)\big(\lambda(s+t)\lambda'(s)\lambda'(t) + \lambda(t)(\lambda'(s)-\lambda'(t))\lambda'(s+t)\big) \bigg] \\
  =: & \frac{\mathcal N(s,t)}{\sqrt{\lambda(s)\lambda(t)}\,(\lambda(s)+\lambda(t))^2}.
\end{align*}

The denominator is smooth and strictly positive for $s>0,~t>0$, hence $\mathcal{H}$ is smooth on $(0,S)\times(0,T)$. Write $\Psi(s,t):=(s+t)\mathcal{H}(s,t)/\sqrt{st}$, so the bound to establish is $\sup_{R_0}|\Psi|<\infty$.

A direct computation, under the assumption $\lambda'(0)>0$, gives,
for each fixed $t>0$,
\begin{align*}
\lim_{s \to 0^+} \Psi(s,t)
&= -\frac{\sqrt{t}\,\sqrt{\lambda'(0)}
\big(\lambda(t)\lambda''(t)-2\lambda'(t)^2
+2\lambda'(0)\lambda'(t)\big)}
{\lambda(t)^{3/2}},
\end{align*}
with the analogous expression, by $\mathcal{H}(t,s)=-\mathcal{H}(s,t)$, for $t\to0^+$ at fixed $s>0$; both are continuous in the free variable on compact subintervals bounded away from $0$.

These edge limits alone do not control $\Psi$ near the origin: taking the two iterated limits $s,t\to0$ in opposite orders gives $\lim_{t\to0^+}\lim_{s\to0^+}\Psi=\lambda''(0)$ but $\lim_{s\to0^+}\lim_{t\to0^+}\Psi=-\lambda''(0)$, so whenever $\lambda''(0)\ne0$ the limit of $\Psi$ at the origin depends on the direction of approach. Ray-wise behavior, even continuous in the ray direction, does not by itself bound $\Psi$ on a full neighborhood of the origin. We close this with an exact factorization instead of a ray/compactness argument.

\smallskip\noindent\emph{Away from the origin.} Fix $\delta_0\in(0,\min(S,T))$ as constructed below, and let $K_{\delta_0}:=\{(s,t)\in R: s+t\ge\delta_0\}$, a compact set not containing the origin. $\Psi$ is continuous on $K_{\delta_0}\cap\{s,t>0\}$, where $\mathcal H$ is smooth and $\sqrt{st}>0$, and it extends continuously to the two edges $\{s=0,\,t\ge\delta_0\}$, $\{t=0,\,s\ge\delta_0\}$ by the limits above, whose $O(s)$ (resp.\ $O(t)$) remainder is uniform once the other variable is bounded away from $0$. Hence $\Psi$ is continuous on all of $K_{\delta_0}$, and $M_H^{\mathrm{far}}:=\sup_{K_{\delta_0}}|\Psi|<\infty$ by compactness.

\smallskip\noindent\emph{Near the origin.} Rescale $\tilde\lambda(z):=\lambda(z)/z$ for $z>0$, $\tilde\lambda(0):=\lambda'(0)$; since $\lambda\in C^3$ near $0$ with $\lambda(0)=0$, $\tilde\lambda$ is continuous on $[0,U^*]$, $U^*:=\min(S,T)$. Substitute $s=\alpha u$, $t=(1-\alpha)u$ ($u\in[0,U^*]$, $\alpha\in[0,1]$) and write
\begin{align*}
a:=& \tilde\lambda(\alpha u),\quad b:=\tilde\lambda((1-\alpha)u),\quad w:=\tilde\lambda(u),\quad p:=\lambda'(\alpha u),\\ 
q:= & \lambda'((1-\alpha)u),\quad r:=\lambda'(u),
\end{align*}
each jointly continuous on $[0,1]\times[0,U^*]$. Substituting into $\mathcal N(s,t)$ above and collecting powers of $u$ — every term is exactly degree $2$ — gives the exact identity
\begin{align*}
\mathcal{N}(s,t) &=u^2\tilde N(\alpha,u), \\
\tilde N(\alpha,u)& :=\alpha^2a^2qr+(1-\alpha)bp\big(wq-(1-\alpha)br\big)-\alpha a\big(wpq+(1-\alpha)b(p-q)r\big),
\end{align*}
with $\tilde N$ jointly continuous on $[0,1]\times[0,U^*]$. 

Direct substitution gives three exact vanishing identities: $\tilde N(0,u)=\tilde N(1,u)=0$ for every $u$, since at $\alpha=0,1$ the bracket collapses via $wq=br$ and $ar=wp$ respectively; and $\tilde N(\alpha,0)=0$ for every $\alpha$, since $a=b=w=p=q=r=\lambda'(0)$ at $u=0$.

Since $\tilde N(\cdot,u)$ vanishes at both endpoints of $[0,1]$ for every $u$, two applications of the mean value theorem give $\tilde N(\alpha,u)=\alpha(1-\alpha)\,\tilde N_b(\alpha,u)$ for a jointly continuous $\tilde N_b$; then $\tilde N(\alpha,0)=0$ forces $\tilde N_b(\alpha,0)=0$ for every $\alpha$, and a further mean-value extraction (using one more derivative of $\lambda$, i.e.\ $\lambda\in C^3$) gives $\tilde N_b(\alpha,u)=u\,\tilde N_c(\alpha,u)$ for a jointly continuous $\tilde N_c$ on $[0,1]\times[0,U^*]$.

The denominator rescales as 
\begin{align*}
\sqrt{\lambda(s)\lambda(t)}(\lambda(s)+\lambda(t))^2& =u^3\sqrt{\alpha(1-\alpha)}\,D(\alpha,u) \\
D(\alpha,u) & :=\sqrt{ab}\,\big(\alpha a+(1-\alpha)b\big)^2, \\
D(\alpha,0) & =\lambda'(0)^3>0
\end{align*} 
for every $\alpha\in[0,1]$. By continuity and using the compactness of $[0,1]$ there is $\delta_0\in(0,U^*)$ with $D\ge\lambda'(0)^3/2$ on $[0,1]\times[0,\delta_0]$ (this is the $\delta_0$ used above). For $0<u\le\delta_0$,
\[
\Psi(s,t)=\frac{u^2\tilde N(\alpha,u)}{u^3\alpha(1-\alpha)D(\alpha,u)}
=\frac{\alpha(1-\alpha)\,u\,\tilde N_c(\alpha,u)}{u\,\alpha(1-\alpha)\,D(\alpha,u)}
=\frac{\tilde N_c(\alpha,u)}{D(\alpha,u)},
\]
the factors of $u$ and $\alpha(1-\alpha)$ cancelling exactly. Since $\tilde N_c,D$ are jointly continuous on the compact rectangle $[0,1]\times[0,\delta_0]$ and $D$ is bounded away from $0$ there, $M_H^{\mathrm{near}}:=\sup_{[0,1]\times[0,\delta_0]}|\tilde N_c/D|<\infty$ by the extreme value theorem.

Taking $M_H:=\max\big(M_H^{\mathrm{far}},M_H^{\mathrm{near}}\big)$ gives $|\Psi|\le M_H$ on all of $R_0$, which is the stated bound. Note that $\Psi$ itself does not extend continuously to the origin in general: along $s=\alpha u$, $t=(1-\alpha)u$ with $u\to0$ at fixed $\alpha\in(0,1)$, $\Psi\to\tilde N_c(\alpha,0)/D(\alpha,0)$, which need not be independent of $\alpha$ — consistent with the direction-dependence found above. Thus $\mathcal{H}$ does not extend continuously to $R$ in general, though it extends as a bounded measurable function, e.g.\ by $\mathcal{H}(\mathbf{0}):=0$.
\end{proof}

We now introduce a function space that will facilitate the analysis of the corrector problem.

\begin{definition}[Lipschitz functions vanishing on the axes]
We define the Banach space
\[
\mathrm{Lip}_0(R) := \Big\{\, x \in C(R) \;\Big|\; x(s,0)=x(0,t)=0 \ \text{for } 0 \leq s \leq S, \ 0 \leq t \leq T \,\Big\},
\]
equipped with the Lipschitz norm
\[
\|x\|_{\mathrm{Lip}} := \sup_{p_1 \ne p_2 \in R} \frac{|x(p_1) - x(p_2)|}{|p_1 - p_2|} < \infty.
\]
Any function $x \in \mathrm{Lip}_0(R)$ satisfies the pointwise bound
\[
|x(s,t)| \leq\|x\|_{\mathrm{Lip}} \min(s,t) \leq \|x\|_{\mathrm{Lip}} \,(s+t).
\]
\end{definition}

Within this setup, the residual $\mathcal{H}$ gives rise to an integrated quantity, which we now show belongs to $\mathrm{Lip}_0(R)$.

\begin{lemma}[Regularity of the Integrated Forcing Term]
Under the hypotheses of Lemma~\ref{lem:regularity-residual}, let $f(s,t)$ denote the forcing term for the integral equation for the corrector, defined by
\[
f(s,t) = - \frac{1}{2\lambda(s+t)} \int_0^s \int_0^t \mathcal{H}(\sigma,\tau)\,d\sigma\,d\tau.
\]
Then $f \in \mathrm{Lip}_0(R)$. That is, $f$ is Lipschitz continuous on $R$ and vanishes on the coordinate axes $s=0$ and $t=0$.
\end{lemma}

\begin{proof}
Consider the integrated residual,
\[
I(s,t) := \int_0^s \int_0^t \mathcal{H}(\sigma,\tau)\,d\sigma\,d\tau.
\]
From the preceding lemma, the residual $\mathcal{H}$ is a bounded measurable function on $R$, with $|\mathcal{H}| \le M_H$ for some constant $M_H>0$. It follows immediately that $I(s,0) = I(0,t) = 0$, and therefore $f(s,0) = f(0,t) = 0$ for all $s,t > 0$.

We begin by obtaining a pointwise bound for $f$. Using the bound on $\mathcal{H}$, we estimate the integral as $$|I(s,t)| \le M_H st.$$ For $s+t > 0$, we have $\lambda(s+t) > 0$. Near the origin, $\lambda(u) \approx \lambda'(0)u$, so for some constant $C_1 > 0$, we have $\lambda(s+t) \ge C_1(s+t)$ on $R$. This gives the bound:
\[
|f(s,t)| = \frac{|I(s,t)|}{2\lambda(s+t)} \le \frac{M_H st}{2C_1(s+t)}.
\]
By the AM-GM inequality, since $2\sqrt{st}\leq s+t$, it follows that $\tfrac{st}{s+t} \leq \tfrac{\sqrt{st}}{2}$. This shows that $|f(s,t)| \le C_2 \sqrt{st}$ for some constant $C_2$, verifying that $f$ is continuous up to the origin with $f(\mathbf{0})=0$.

To prove Lipschitz continuity, use the absolute continuity of $I$ on coordinate lines. Its weak partial derivatives exist almost everywhere and satisfy
\[
I_s(s,t)=\int_0^t\mathcal H(s,\tau)\,d\tau,
\qquad
I_t(s,t)=\int_0^s\mathcal H(\sigma,t)\,d\sigma,
\]
so $|I_s(s,t)|\le M_Ht$ and $|I_t(s,t)|\le M_Hs$. The quotient rule therefore holds almost everywhere; in particular,
\[
f_s(s,t) = -\frac{I_s(s,t)\,\lambda(s+t) - I(s,t)\,\lambda'(s+t)}{2\lambda(s+t)^2}.
\]
Since $\lambda'(0)>0$ and $\lambda\in C^1$, there are positive constants $m,M$ and a neighborhood of the origin on which $m(s+t)\le\lambda(s+t)\le M(s+t)$ and $|\lambda'(s+t)|\le M$. Hence the displayed bounds for $I$ and $I_s$ give $|f_s|\le C$ almost everywhere near the origin; the same calculation gives $|f_t|\le C$. Away from the origin, the corresponding bounds are uniform by compactness. Thus $f\in W^{1,\infty}(R)$, and its continuous representative is Lipschitz on $R$.
\end{proof}

We now prove the existence of the corrector solution, showing that it is more regular than the parametrix $x_p$, which is discontinuous at $\mathbf{0}$.

\textbf{Corner-scale behavior of the kernel.}\label{para:volterra-fails} For a fixed point $(s,t)$ with $s+t>0$, the kernels below remain finite as the integration variables approach their endpoints. The danger is at the corner of the domain, not the diagonal: since $\lambda(u)=\lambda'(0)u+o(u)$ near $u=0$, both $\lambda'(\sigma+t)$ and $\lambda'(s+\tau)$ tend to the constant $\lambda'(0)$ while $\lambda(s+t)$ vanishes linearly, so near the corner
\[
K_1(s,t,\sigma)\approx\frac{1}{2(s+t)},\qquad K_2(s,t,\tau)\approx\frac{1}{2(s+t)}.
\]
The kernels themselves blow up like $1/(s+t)$ as the corner is approached so there are no constant bounds $K_1,K_2$ on $R$ -- an impediment to using standard Volterra estimates as in \S\ref{sec:cauchy-goursat}.

\subsection{Existence and Uniqueness for the Corrector}\label{sec:corrector}
The existence and uniqueness of the corrector can be expressed through an associated Volterra integral equation. Specifically, we consider
\begin{equation}\label{eq:Volterra}
x(s,t) = \mathcal{F}(s,t) + \int_{0}^{s} K_1(s,t,\sigma)\, x(\sigma,t) \, d\sigma
         + \int_{0}^{t} K_2(s,t,\tau)\, x(s,\tau) \, d\tau,
\end{equation}
for an unknown function $x \in C_0(R)$, 
where $R = [0, S] \times [0, T]$ and $C_0(R)$ denotes the Banach space of continuous functions on $R$ that vanish on the axes $s=0$ and $t=0$. 

The kernels are defined for $s+t>0$ by
\[
K_1(s,t,\sigma) = \frac{\lambda'(\sigma+t)}{2\lambda(s+t)}, 
\qquad 
K_2(s,t,\tau) = \frac{\lambda'(s+\tau)}{2\lambda(s+t)}.
\]

\begin{theorem}\label{thm:corrector}
Let the kernels $K_1$ and $K_2$ be defined as above. Suppose $\lambda$ satisfies:
\begin{enumerate}
    \item $\lambda(0) = 0$;
    \item $\lambda\in C^1([0,S+T])$ and $\lambda'(u)>0$ for all $u\in[0,S+T]$.
\end{enumerate}
Then, for any $\mathcal{F} \in C_0(R)$, the Volterra equation \eqref{eq:Volterra} admits a unique solution $x \in C_0(R)$.
\end{theorem}

\begin{proof}
The proof uses the Banach Fixed-Point Theorem. We rewrite \eqref{eq:Volterra} in the form
\[
x = \mathcal{A}(x), \quad \mathcal{A}(x) := \mathcal{F} + \mathcal{K}x,
\]
where $\mathcal{K}$ is the linear integral operator defined in \eqref{eq:K_op} by the kernels $K_1$ and $K_2$.

\medskip
\noindent \textbf{Preservation of boundary conditions.}
The operator $\mathcal{A}$ preserves the property of vanishing on the axes. Indeed, if $x \in C_0(R)$ satisfies $x(s,0)=0$ and $x(0,t)=0$, then
\begin{align*}
(\mathcal{A}x)(s,0) &= \mathcal{F}(s,0)+ \int_0^s K_1(s,0,\sigma)\,x(\sigma,0)\,d\sigma = 0,\\[0.5em]
(\mathcal{A}x)(0,t) &= \mathcal{F}(0,t)+\int_0^t K_2(0,t,\tau)\,x(0,\tau)\,d\tau = 0.
\end{align*}
Thus $\mathcal{A}$ preserves the axis conditions. Away from the origin $\mathbf{0}$, the kernels $K_1$ and $K_2$ are continuous in $(s,t)$, so $\mathcal{A}x$ is continuous there.

\medskip
\noindent \textbf{Behavior at the origin.} 
The kernels are well defined for $s+t>0$, but the denominator vanishes at the origin. Set $(\mathcal Ax)(\mathbf0)=0$. For $x\in C_0(R)$, define
\[
\omega_x(\delta):=\sup\{|x(a,b)|:(a,b)\in R,\ a+b\le\delta\}.
\]
Continuity and the axis condition give $\omega_x(\delta)\to0$ as $\delta\downarrow0$. If $s+t\le\delta$, every point sampled by the two Volterra integrals also has coordinate sum at most $\delta$. Since the kernels are nonnegative and their total integrated mass is at most $1$ by the exact identities below,
\[
|(\mathcal Kx)(s,t)|\le\omega_x(\delta).
\]
Together with $\mathcal F(s,t)\to0$, this proves $(\mathcal Ax)(s,t)\to0$ and hence continuity at the origin. Thus $\mathcal A$ maps $C_0(R)$ into itself.

\medskip
\noindent \textbf{Contraction property.}
We now estimate the operator norm on $C_0(R)$ with the supremum norm $\|\cdot\|_\infty$:
\begin{equation}\label{eq:Knorm}
\|\mathcal{K}\|_\infty = \sup_{(s,t)\in R} \left\{ \int_0^s |K_1(s,t,\sigma)| \, d\sigma 
+ \int_0^t |K_2(s,t,\tau)| \, d\tau \right\}.
\end{equation}

From the hypothesis, the rescaled function
\begin{equation}\label{eq:defg}
\tilde\lambda(u) := \frac{\lambda(u)}{u}, \quad u>0,
\end{equation}
(the same rescaling used in the proof of the Regularity of the Residual lemma) can be extended
continuously to $[0, S+T]$ by $\tilde\lambda(0):=\lambda'(0)$. Since $\tilde\lambda$ is continuous
on this compact interval, the Extreme Value Theorem guarantees that the following are attained:
\[
m_0 := \inf_{u \in [0, S+T]} \tilde\lambda(u),
\qquad
M := \sup_{u \in [0, S+T]} \tilde\lambda(u).
\]
By the Mean Value Theorem, for every $u > 0$ there exists $\xi \in (0,u)$ such that $\tilde\lambda(u) = \lambda'(\xi)$, and thus $\tilde\lambda(u) > 0$. At $u=0$ we have $\tilde\lambda(0) = \lambda'(0) > 0$. Consequently, $m_0 > 0$. In addition, since $\tilde\lambda$ is continuous on a compact set, we also have $M < \infty$.

Because $\lambda'(u)>0$ and $\lambda(u)=u\tilde\lambda(u)>0$ for $u>0$, the kernels are nonnegative. Hence the operator norm \eqref{eq:Knorm} becomes
\[
\|\mathcal{K}\|_\infty = \sup_{(s,t)\in R_0} 
\left\{ \int_0^s K_1(s,t,\sigma)\, d\sigma 
+ \int_0^t K_2(s,t,\tau) \, d\tau \right\}.
\]

Evaluating the integrals yields
\begin{align*} 
\int_0^s K_1(s,t,\sigma)\, d\sigma 
&= \frac{\lambda(s+t) - \lambda(t)}{2\lambda(s+t)}, \\[0.5em]
\int_0^t K_2(s,t,\tau)\, d\tau 
&= \frac{\lambda(s+t) - \lambda(s)}{2\lambda(s+t)},
\end{align*}
and their sum becomes
\[
1 - \frac{\lambda(s) + \lambda(t)}{2\lambda(s+t)}.
\]

In terms of $\tilde\lambda$ given by \eqref{eq:defg}, we have
\[
\frac{\lambda(s) + \lambda(t)}{2\lambda(s+t)}
= \frac{s \tilde\lambda(s) + t \tilde\lambda(t)}{2(s+t) \tilde\lambda(s+t)}
\;\ge\; \frac{m_0}{2M}.
\]
Thus,
\[
\int_0^s K_1\,d\sigma + \int_0^t K_2\,d\tau \;\le\; 1 - \frac{m_0}{2M},
\]
for all $(s,t)\in R$. Hence
\[
\|\mathcal K\|_\infty \le 1 - \frac{m_0}{2M} < 1.
\]

Therefore $\mathcal{K}$ is a contraction on $C_0(R)$. By the Banach Fixed-Point Theorem, $\mathcal{A}$ has a unique fixed point, which is the unique solution to \eqref{eq:Volterra}. Since $x\in C_0(R)$, this solution vanishes on the coordinate axes as well.
\end{proof}

Under the combined hypotheses of Lemma~\ref{lem:regularity-residual} and Theorem~\ref{thm:corrector}, the integrated residual forcing belongs to $\mathrm{Lip}_0(R)\subset C_0(R)$. Consequently, the parametrix admits a unique continuous corrector vanishing on the axes.

\subsection{The Inverse Hodograph Map and Fold Singularities}\label{sec:inverse}

On a strictly partially convex classical patch, the transformation from physical coordinates $(x,y)$ to characteristic coordinates $(s,t)$ converts the nonlinear Monge--Amp\`ere equation into the linear formulation above. This forward transformation is valid where $w_{xx}$ has a fixed nonzero sign.

The converse is conditional. Reconstructing a physical solution requires sufficient differentiability and exactness for the differential-form reconstruction and, for each fixed $u$, injectivity of $p\mapsto x(u,p)$ with $x_p$ of fixed nonzero sign. A continuous Volterra solution need not satisfy these conditions. In particular, inversion of the hodograph transformation can fail.
The following example shows concretely how this failure can occur.

\begin{example}[Birth of a Singularity from Weakly Compatible Data]\label{ex:fold-singularity}
Consider again the case $\lambda(u) = u$ — the same model as the running product solution $w_0 = e^{-y}\cos x$ of \S\ref{subsec:geometric_setting} (where $\lambda(y)=e^{-y}$ gives the timelike coordinate $u(y) = e^{-y}$, i.e.\ $\lambda(u)=u$) and the arcsine boundary trace of Example~\ref{ex:arcsin}: all three examples in this paper are instances of the single equation
$u(x_{uu}-x_{pp})+x_u=0$, not three unrelated special cases. For this $\lambda$, the associated linear PDE is the Euler--Poisson--Darboux equation, $u(x_{uu} - x_{pp}) + x_u = 0$.
An exact polynomial solution for this is
\[
x_{\text{poly}}(u,p) = -\tfrac{5 \pi}{26} \left[(3u^4 - 12u^2 + 10)p + 8(u^2-1)p^3 + \tfrac{8}{5}p^5\right],
\]
with derivative with respect to $p$ given by
\[
(x_{\text{poly}})_p(u,p) = -\tfrac{5 \pi}{26} \left[(3u^4 - 12u^2 + 10) + 24(u^2-1)p^2 + 8p^4\right].
\]
Consider now the domain $\Omega_1 = \{(s,t) \mid s,t \ge 0, s+t \ge 1\}$. We pose a Cauchy–Goursat problem with the boundary data defined as follows.
\begin{itemize}
    \item \textbf{Cauchy Data on $u=1$:} For $p \in [-1,1]$, we prescribe boundary data that agree with those of the exact polynomial solution
    \begin{align*}
    \text{Dirichlet data:} \quad  x(1,p) &= -\tfrac{\pi}{26}(5p + 8p^5), \\[0.5em]
    \text{Normal derivative:} \quad  x_u(1,p) &= -\tfrac{10\pi}{13} \, p(4p^2 - 3).
    \end{align*}
  The concavity condition is then satisfied along the Cauchy boundary $u=1$, since 
$(x_{\text{poly}})_p(1,p) = -\tfrac{5\pi}{26}(1 + 8p^4) < 0$ for all $p \in [-1,1]$.
    
    \item \textbf{Goursat Data on $p=\pm u$:} For $u > 1$, we prescribe constant Goursat data, taking the values from the corners of the Cauchy segment:
    \begin{align*}
        x(u,u) &= x_{\text{poly}}(1,1) = -\tfrac{5\pi}{26}\left(1 + \tfrac{8}{5}\right) = -\tfrac{\pi}{2}, \\[0.5em]
        x(u,-u) &= x_{\text{poly}}(1,-1) = \tfrac{5\pi}{26}\left(1 + \tfrac{8}{5}\right) = \tfrac{\pi}{2}.
    \end{align*}
\end{itemize}
This defines a well-posed problem. At the corners where the boundary type changes, for instance, at $(u,p)=(1,1)$, the Dirichlet data are continuous by construction. The derivatives, however, do not match: along the Goursat boundary $p=u$ the tangential derivative vanishes, whereas along the Cauchy segment $u=1$ we have 
$(x_{\text{poly}})_u(1,1) + (x_{\text{poly}})_p(1,1) = -\tfrac{5\pi}{2}$. 

Thus, the data is not $C^1$-compatible. It is, however, weakly compatible, fitting precisely into the framework we have developed for such solutions. The unique solution in the forward domain of dependence of the Cauchy data,
\[
\mathcal D=\{(u,p):1\leq u\leq2,\quad |p|\leq2-u\},
\]
is given by the polynomial $x(u,p)=x_{\text{poly}}(u,p)$.
Its slanted sides are incoming characteristics, while the dashed extensions
$p=\pm u$ are outgoing characteristics from the two corners, carrying the
prescribed values $x(u,u)=-\pi/2$ and $x(u,-u)=\pi/2$.

We now demonstrate the ``birth'' of a singularity within this region.  
At the interior point $(u^\ast,p^\ast)=(\tfrac{3}{2},0)$, we find $x_p(u^\ast,p^\ast)=\tfrac{145\pi}{416}$. Initially, $x_p<0$ everywhere on the boundary at $u=1$, but inside the domain it evolves to $x_p>0$ at this point. Thus $x_p$ crosses zero and the inverse hodograph map develops a fold \cite{brander2017pseudospherical}. Figure~\ref{fig:fold-singularity}(b) makes this loss of invertibility visible: on the common interval $-1/2\leq p\leq1/2$, the profile at $u=3/2$ develops two turning points. In particular, $x_{\mathrm{poly}}(3/2,p)=0$ has three distinct solutions, so the inverse slope is not single-valued; no reconstructed physical surface is asserted for this folded state. Outside $\mathcal D$ the polynomial no longer represents the solution, but the solution itself remains continuous, as guaranteed by Theorem~\ref{thm:existence-uniqueness}.
\end{example}

\begin{figure}[htbp]
\centering
\begin{minipage}[b]{0.49\linewidth}
\centering
\includegraphics[width=\linewidth]{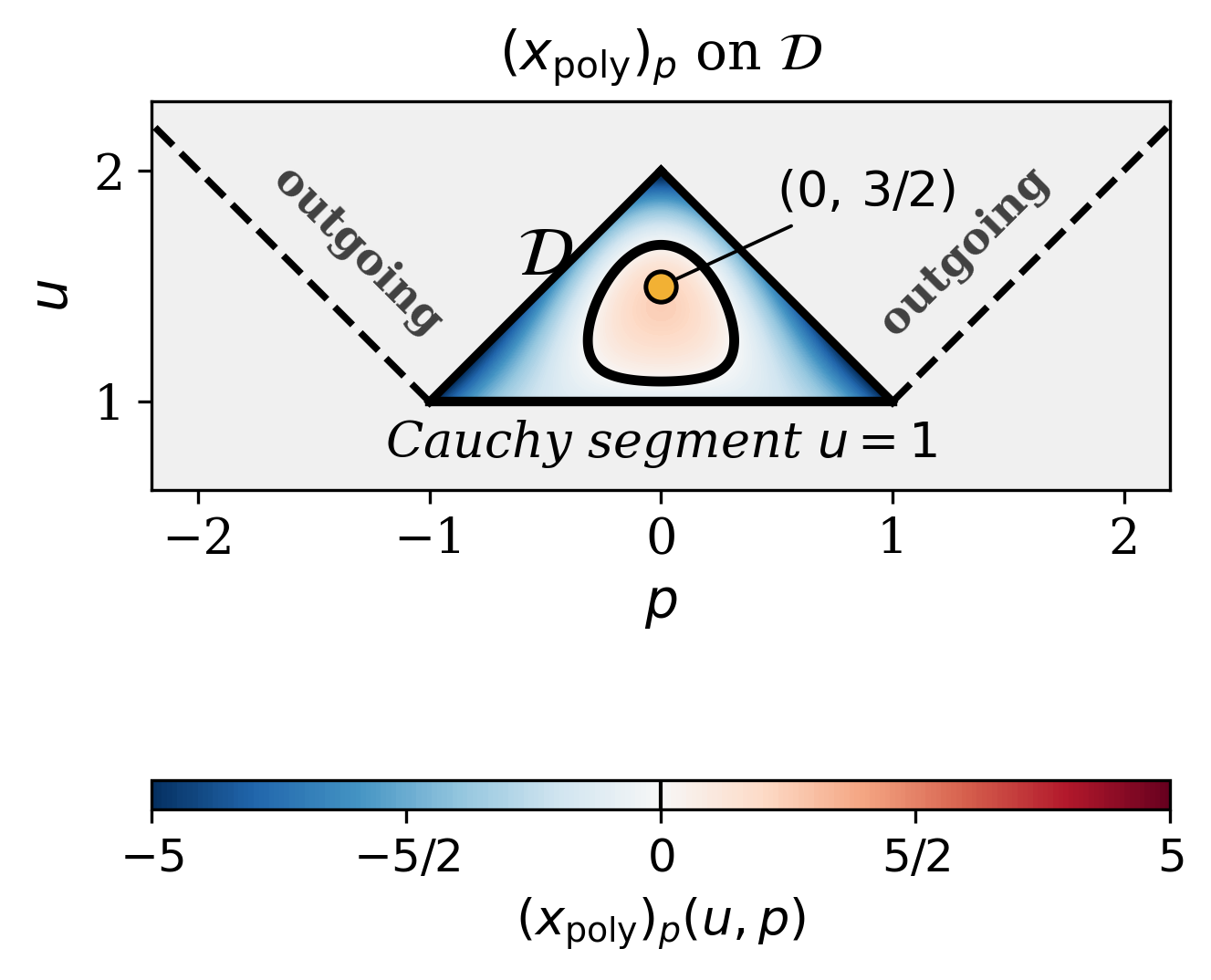}
\smallskip
\textbf{(a)}
\end{minipage}\hfill
\begin{minipage}[b]{0.49\linewidth}
\centering
\includegraphics[width=\linewidth]{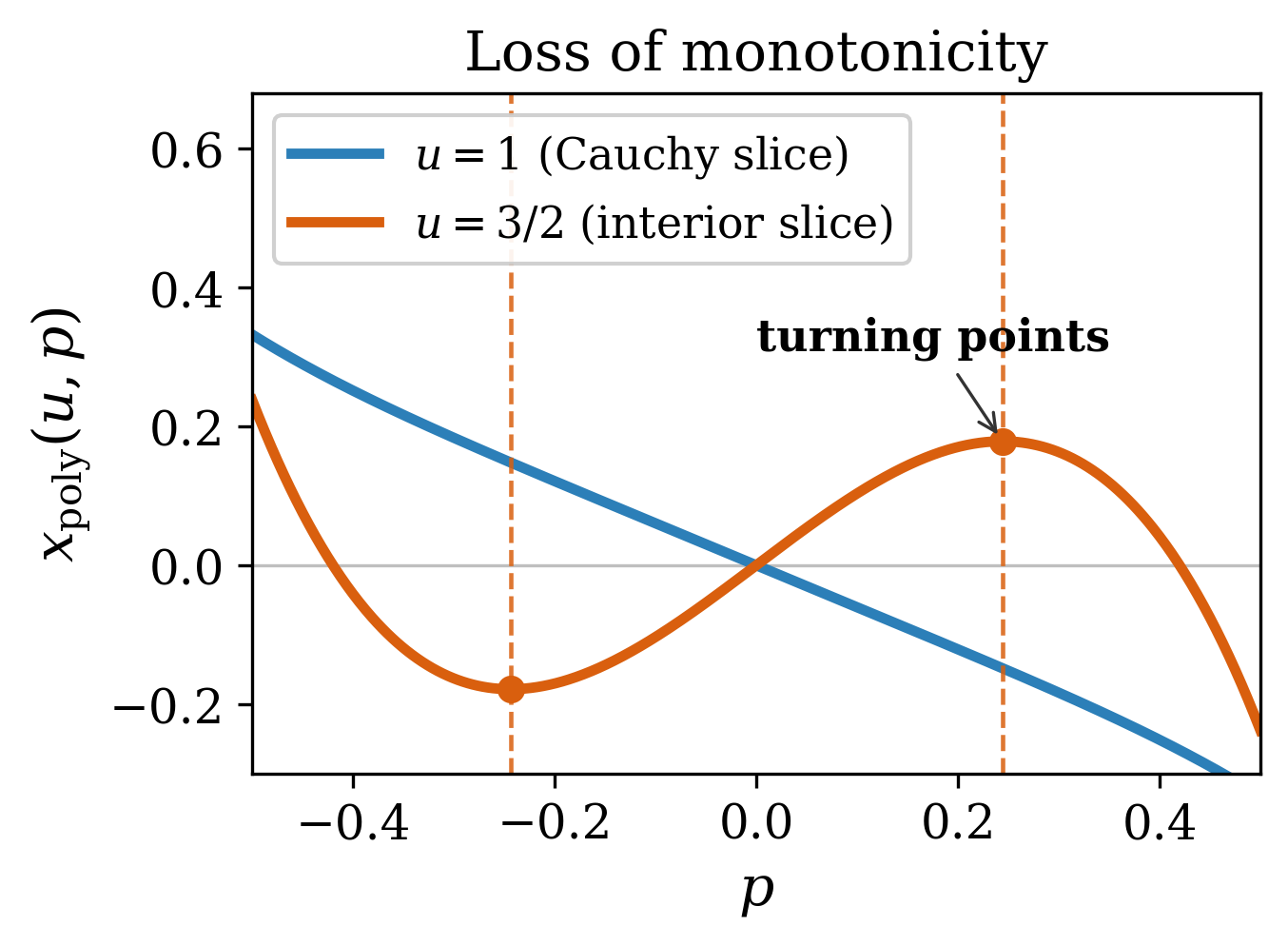}
\smallskip
\textbf{(b)}
\end{minipage}
\caption{Fold formation in Example~\ref{ex:fold-singularity}. (a) Sign of $(x_{\mathrm{poly}})_p$ on $\mathcal D$; the heavy black curve is the fold locus $(x_{\mathrm{poly}})_p=0$. (b) Profiles $p\mapsto x_{\mathrm{poly}}(u,p)$ at $u=1$ and $u=3/2$ on the common interval $-1/2\le p\le1/2$.}\label{fig:fold-singularity}
\end{figure}
 
\section{Energy Estimates and Stability for the Cauchy–Goursat Problem}\label{sec:estimates}

A central tool for analyzing hyperbolic problems is the use of energy identities \cite{EvansbookPDE}. We derive an estimate for fields with homogeneous Goursat data, balancing growth against dissipation and forcing. Its two applications are a conditional first-variation estimate under curvature perturbations and, for fixed curvature, a secant estimate for differences within a regular singular-solution class.

\subsection{Energy Identity and Estimates}

For a field $y$ that is $C^1$ where $s,t>0$, define its slice energy along
$\Gamma_u=\{s+t=u\}$ by
\begin{equation}\label{eq:Eu}
E[u;y]:=\int_0^u\lambda(u)\Big(y_s(s,u-s)^2+y_t(s,u-s)^2\Big)\,ds\ \in[0,\infty].
\end{equation}

\begin{remark}[Connection to the Cauchy--Goursat boundary data]\label{rem:Eu-weak-compatible}
Each slice $\Gamma_u$ is geometrically identical to the Cauchy segment $\Gamma_C$ of
\S\ref{sec:cauchy-goursat}, just at level $u$: writing $y_s-y_t=\frac{d}{ds}[y(s,u-s)]$ for the
tangential derivative and $y_s+y_t=(\partial_s+\partial_t)y$ for the normal one (as in
\eqref{prob:cauchy-goursat}), $y_s^2+y_t^2=\tfrac12[(y_s+y_t)^2+(y_s-y_t)^2]$ turns $E[u;y]$ into
$\tfrac{\lambda(u)}2$ times the sum of the $L^2(0,u)$ norms of these two derivatives. At $u=c$ this
is exactly \S\ref{sec:cauchy-goursat}'s data: if $y$ attains weakly compatible data $\{g,h,f,n\}$
(Definition~\ref{defn:weak-compatible}), then
\[
E[c;y]=\frac{\lambda(c)}{2}\Big(\|n\|_{L^2(\Gamma_C)}^2+\|f'\|_{L^2(\Gamma_C)}^2\Big),
\]
strictly stronger than weak compatibility, which only asks $f',n\in L^1(\Gamma_C)$. The arcsine
example (Example~\ref{ex:arcsin}) sits exactly on this gap: $f',n\in L^1(\Gamma_C)\setminus
L^2(\Gamma_C)$ there, so $E[c;x_a]=\infty$ -- a reflection of  the $s^{-1/2}$ (or $t^{-1/2}$)  divergence of $x_p$ (See \eqref{eq:local_singularity}) giving $E[u;x_p]=\infty$ for every $u>0$.
\end{remark}

Theorem~\ref{thm:corrector} gives a unique pure-Goursat corrector on a rectangle containing the singular corner. A different class arises on a truncated cone $c\le s+t\le C$: regular Cauchy traces may vary while the Goursat trace and the coefficient $\lambda$ remain fixed. Any two members of this truncated-cone class share the parametrix $x_p$ from \eqref{eq:x_p}. Thus
\[
\delta x=x^{(1)}-x^{(2)}=x_{\mathrm{corr}}^{(1)}-x_{\mathrm{corr}}^{(2)}
=:\delta x_{\mathrm{corr}}
\]
has homogeneous Goursat data. The energy estimate below is designed for this difference, not for
the full singular solutions, which have infinite ``bare energy'' (Remark~\ref{rem:Eu-weak-compatible}).

We therefore begin with the inhomogeneous problem
\begin{equation}\label{eq:PDEforEstimates}
2\lambda(s+t)x_{st} + \lambda'(s+t)(x_s+x_t) = \rho(s,t)
\end{equation}
on the domain $\Omega= \{(s,t) \mid s,t \geq 0,\; s+t \geq c\}$,
with $\lambda \in C^1$, $\lambda > 0$, and $\lambda' > 0$ for $s+t>0$.
The solution $x(s,t)$ is assumed to satisfy the homogeneous Goursat boundary conditions
\[
x(s,0)=0 ~~\text{for } s \geq c, \quad x(0,t)=0 ~~\text{for } t \geq c.
\]

\begin{proposition} \label{pro:Energy}
Let $x \in C^1(\Omega)$ with $x_{st} \in C(\Omega)$ solve the problem above. Then for any $u>c$, the solution satisfies the energy identity
\begin{align}\label{energyid}
E[u;x] = E[c;x] + \int_c^u \int_0^v \Big( \rho(x_s+x_t) - 2\lambda'(v)x_s x_t \Big) \,ds\,dv,
\end{align}
where the integrands on the right are evaluated at $(s, v-s)$.
\end{proposition}

\begin{proof}
The proof proceeds by first deriving a local conservation law and then integrating it over the domain.

\medskip
\noindent \textbf{Step 1: Derivation of the Conservation Law.}
Multiplying \eqref{eq:PDEforEstimates} by $(x_s+x_t)$, rewriting $2\lambda x_{st}(x_s+x_t)$ as
$\partial_t(\lambda x_s^2)+\partial_s(\lambda x_t^2)-\lambda'(x_s^2+x_t^2)$, and expanding
$(x_s+x_t)^2$ gives directly the conservation law
\begin{equation} \label{eq:conservation_inhomogeneous}
\frac{\partial}{\partial t}\left(\lambda x_s^2\right) + \frac{\partial}{\partial s}\left(\lambda x_t^2\right) + 2\lambda'x_s x_t = \rho(x_s+x_t).
\end{equation}

\medskip
\noindent \textbf{Step 2: Integration over the Domain.}
We now integrate \eqref{eq:conservation_inhomogeneous} over the domain $D_u = \{(s',t') \mid c \le s'+t' \le u, s' \ge 0, t' \ge 0\}$. By Green’s theorem, the divergence term contributes boundary integrals. The characteristic sides vanish due to homogeneous Goursat conditions, while the lines $s+t=u$ and $s+t=c$ contribute $E[u;x]$ and $-E[c;x]$, respectively. Thus, the integral of the divergence part is $E[u;x]-E[c;x]$.

Equating the integrated parts gives
\[
E[u;x] - E[c;x] + \iint_{D_u} 2\lambda'x_s x_t \,ds'dt' = \iint_{D_u} \rho(x_s+x_t)\,ds'dt'.
\]
Rearranging and writing the integrals in iterated form yields the final identity \eqref{energyid}.
\end{proof}

\begin{corollary}[Energy Estimate for the Solution]\label{cor:energyestimate}
Let $x$ be a solution to the inhomogeneous problem on the domain
\[
\Omega_u = \{(s,\zeta-s) \mid c \leq \zeta \leq u, \; 0 \leq s \leq \zeta \}.
\]
If $\lambda'(\zeta) > 0$ for all $\zeta \geq c$, then the energy satisfies the bound
\begin{equation}\label{eq:Energy_Corrected}
E[u;x] \leq \frac{\lambda(u)}{\lambda(c)}\exp\!\left(\int_c^u \frac{dv}{\lambda(v)}\right)\left(E[c;x] + \frac{1}{2}\|\rho\|_{L^2(\Omega_u)}^2\right), \quad u > c,
\end{equation}
where $$\|\rho\|_{L^2(\Omega_u)}^2 = \int_c^u \int_0^{\zeta} \rho(s,\zeta-s)^2 \, ds \, d\zeta.$$
\end{corollary}

\begin{proof}
We start from the energy identity \eqref{energyid}. We bound the terms on the right-hand side using the Cauchy--Schwarz and Young inequalities:
\begin{align*}
\left| \rho(x_s+x_t) \right| &\leq \frac{1}{2}\rho^2 + \frac{1}{2}(x_s+x_t)^2 \leq \frac{1}{2}\rho^2 + x_s^2 + x_t^2, \\[0.5em]
\left| -2\lambda' x_s x_t \right| &\leq \lambda'(x_s^2+x_t^2).
\end{align*}
Substituting these inequalities into \eqref{energyid} yields
\begin{align*}
E[u;x] &\leq E[c;x] + \int_c^u \int_0^v \left( \frac{1}{2}\rho^2 + x_s^2 + x_t^2 \right) ds\,dv + \int_c^u \int_0^v \lambda'(v)(x_s^2+x_t^2)\,ds\,dv \\[0.5em]
&\leq E[c;x] + \frac{1}{2}\|\rho\|_{L^2(\Omega_u)}^2 + \int_c^u \int_0^v (1+\lambda'(v))(x_s^2+x_t^2)\,ds\,dv.
\end{align*}
Using the definition of $E[v;x]$ in \eqref{eq:Eu}, we can write the last term as:
\[
E[u;x] \leq E[c;x] + \frac{1}{2}\|\rho\|_{L^2(\Omega_u)}^2 + \int_c^u \frac{1+\lambda'(v)}{\lambda(v)} E[v;x] \,dv.
\]
Since $\int_c^u \frac{1+\lambda'(v)}{\lambda(v)}\,dv = \int_c^u \frac{dv}{\lambda(v)} + \log\frac{\lambda(u)}{\lambda(c)}$,
Gronwall's lemma gives the bound \eqref{eq:Energy_Corrected} directly.
\end{proof}

\subsection{Stability under Perturbations of the Curvature Function}

We next derive a conditional linearized estimate for finite-energy solutions under perturbations of the curvature function $\lambda(u)$. We assume that a corresponding solution family exists and is differentiable in the perturbation parameter, with the complete transformed boundary data held fixed. The argument estimates its first variation; it does not assert nonlinear continuous dependence of the curvature-to-solution map.

\begin{lemma}[Linearized Equation for Finite-Energy Solutions] \label{le:linearizing} Recall that the governing equation can be written in terms of the operator
\[
L_\lambda[x] := 2\lambda(s+t)x_{st} + \lambda'(s+t)(x_s+x_t),
\]
so that solutions satisfy $L_\lambda[x]=0$.
Let $x_0$ be a finite-energy solution of the (unforced) Cauchy--Goursat problem for the
coefficient function $\lambda_0$, with its complete prescribed Cauchy--Goursat data held fixed
under the perturbation below.
Consider a perturbed coefficient of the form
\[
\lambda(u;\epsilon) = \lambda_0(u) + \epsilon \lambda_1(u),
\]
with corresponding solution $x(s,t;\epsilon)$, assumed differentiable at $\epsilon=0$.
Then the regular first-order variation
\[
x_1 := \left.\frac{\partial}{\partial \epsilon} x(s,t;\epsilon)\right|_{\epsilon=0}
\]
satisfies the linear inhomogeneous equation
\[
L_{\lambda_0}[x_1] = \rho(s,t).
\]
The forcing term $\rho$ is given by
\[
\rho(s,t) = -2 \lambda_0(u) \, \frac{d}{du}\!\big(\delta\log \lambda\big) \, (x_0)_u,\quad \text{with} \quad
\delta\log\lambda := \frac{\lambda_1(u)}{\lambda_0(u)}.
\]
\end{lemma}

\begin{proof}
Differentiating $L_\lambda[x]=0$ at $\epsilon=0$ and eliminating $(x_0)_{st}$ with
$L_{\lambda_0}[x_0]=0$ gives
\[
L_{\lambda_0}[x_1] = -\left(\lambda_1' - \lambda_1\frac{\lambda_0'}{\lambda_0}\right)\big((x_0)_s+(x_0)_t\big)
= -\lambda_0 \frac{d}{du}\!\left(\frac{\lambda_1}{\lambda_0}\right)\big((x_0)_s+(x_0)_t\big),
\]
which is the stated formula because $x_u=\tfrac12(x_s+x_t)$ in the $(u,p)$ variables. Fixed
complete Cauchy--Goursat data give homogeneous data for $x_1$.
\end{proof}

We now apply the energy functional from Proposition~\ref{pro:Energy} to this finite-energy base
solution and its regular first-order variation. Write
\[
E_0(u):=E[u;x_0],\qquad E_1(u):=E[u;x_1].
\]

\begin{proposition}[Linear Stability of Finite-Energy Solutions]\label{prop:linear-stability-finite-energy}
Let $x_0$ be a $C^2$ finite-energy solution for a smooth coefficient $\lambda_0(u)$ satisfying $\lambda_0>0$ and $\lambda_0'>0$ on $[c,C]$. Assume a corresponding solution family exists and is differentiable at the base coefficient, with its complete transformed Cauchy--Goursat data held fixed under a perturbation $\delta\lambda$. Then the energy
$E_1(u)$ of the first-order variation $x_1$ satisfies
\[
\sup_{u \in [c,C]} E_1(u) \leq M  \left(\sup_{u \in [c,C]} E_0(u)\right)
\left\| \frac{d}{du}\left(\frac{\delta\lambda}{\lambda_0}\right) \right\|_{L^2(\lambda_0(u)du)}^2,
\]
for a constant $M$ that depends on $\lambda_0$ and the interval $[c,C]$ but not on the perturbation $\delta\lambda$.
\end{proposition}

\begin{proof}
By Lemma~\ref{le:linearizing}, $x_1$ has homogeneous data, hence $E_1(c)=0$.
Corollary~\ref{cor:energyestimate}, applied on $[c,C]$, gives
\[
\sup_{u \in [c,C]} E_1(u) \leq M_1 \|\rho\|_{L^2(\Omega_C)}^2,
\]
where $M_1$ absorbs the bounded factor in \eqref{eq:Energy_Corrected}. The formula for $\rho$ and
$(a+b)^2\leq2(a^2+b^2)$ give
\begin{align*}
\|\rho\|_{L^2(\Omega_C)}^2
&\le \int_c^C \left(\lambda_0(v) \frac{d}{dv}(\delta\log\lambda)\right)^2
\frac{2E_0(v)}{\lambda_0(v)}\,dv \\
&\le 2 \left(\sup_{w \in [c,C]} E_0(w)\right)  \int_c^C \lambda_0(v) \left(\frac{d}{dv}(\delta\log\lambda)\right)^2 dv.
\end{align*}
Combining these inequalities gives the assertion with $M=2M_1$.
\end{proof}

\begin{remark}[Invariance and Orthogonal Perturbations]
The governing equation is invariant under a uniform scaling of the coefficient, $\lambda \mapsto k\lambda$.
Our estimate reflects this correctly: if $\delta\log\lambda$ is constant, then its derivative vanishes, the forcing term $\rho$ is zero, and the first variation $x_1$ vanishes. Thus the bound depends only on the derivative of the relative perturbation and is insensitive to its additive constant, as required by the scaling invariance.
\end{remark}

\begin{remark}[The singular family is excluded]\label{rem:scope-finite-energy}
The finite-energy hypothesis in Proposition~\ref{prop:linear-stability-finite-energy} is a genuine
restriction, not merely a technicality: singular states of the form $x_p+x_{\mathrm{corr}}$
do not, in general, belong to it, since the parametrix $x_p$ carries
infinite bare energy \eqref{eq:Eu}.  
\end{remark}
This motivates the following discussion of a \emph{renormalized energy} appropriate to these    ``singular" solutions.

\subsection{Renormalized Energy and the Singular Solution Family}\label{subsec:renormalized-energy}

Since any two singular solutions at fixed $\lambda$ share the
parametrix $x_p$, their difference is a genuine finite-energy field with homogeneous Goursat data,
governed by the same operator $L_\lambda$ as above. This motivates
renormalizing -- tracking the energy of $x-x_p$ rather than of $x$ itself -- and comparing two
members of the resulting regular class. The remark and corollary below give this fixed-curvature
secant estimate on a truncated cone $c\le s+t\le C$ bounded away from the corner.

\begin{remark}[Renormalized-energy singular states]\label{rem:energy-scope}
For a singular state $x=x_p+x_{\mathrm{corr}}$, define its renormalized slice energy by
\[
\widetilde E[u;x]:=E[u;x-x_p]=E[u;x_{\mathrm{corr}}].
\]
For fixed $\lambda$ and $0<c<C$, let $\mathcal S_{c}^{\mathrm{ren}}(\lambda)$ be the set of singular
solutions on $c\leq s+t\leq C$ with the prescribed Goursat trace whose correctors have the
regularity of Proposition~\ref{pro:Energy} and satisfy $\widetilde E[c;x]<\infty$. This is an
additional regular class; neither Definition~\ref{defn:hodograph-weak} nor
Theorem~\ref{thm:corrector} alone supplies the energy regularity required here.

Whenever $\mathcal S_{c}^{\mathrm{ren}}(\lambda)$ is nonempty, it is an affine space. Indeed, for every $x_\ast\in\mathcal S_{c}^{\mathrm{ren}}(\lambda)$,
\[
\mathcal S_{c}^{\mathrm{ren}}(\lambda)=x_\ast+\mathcal V_{c}(\lambda),
\]
where $\mathcal V_{c}(\lambda)$ consists of the finite-energy, homogeneous-Goursat solutions
$z$ of $L_\lambda[z]=0$ with $E[c;z]<\infty$. Thus the renormalized energy, rather than the
energy of the full singular field, is the natural state variable for this family.
\end{remark}

The affine-space observation above and Corollary~\ref{cor:energyestimate} give the following
secant stability within $\mathcal S_{c}^{\mathrm{ren}}(\lambda)$.

\begin{corollary}[Secant stability of the singular solutions]\label{cor:affine-secant}
Let $x^{(1)},x^{(2)}\in\mathcal S_{c}^{\mathrm{ren}}(\lambda)$. Then
$\delta x := x^{(1)}-x^{(2)} = x_{\mathrm{corr}}^{(1)}-x_{\mathrm{corr}}^{(2)}$
satisfies $L_\lambda[\delta x]=0$ with homogeneous Goursat data and $E[c;\delta x]<\infty$, so
\[
E[u;\delta x] \le \frac{\lambda(u)}{\lambda(c)}\exp\!\left(\int_{c}^u\frac{dv}{\lambda(v)}\right) E[c;\delta x], \qquad c<u\leq C.
\]
\end{corollary}

\begin{proof}
The shared parametrix cancels in $\delta x$; each $x_{\mathrm{corr}}^{(i)}$ vanishes on the
Goursat boundary, so $\delta x$ does too, and
$E[c;\delta x] \le 2\big(\widetilde E[c;x^{(1)}]+\widetilde E[c;x^{(2)}]\big)<\infty$.
Linearity of $L_\lambda$ gives $L_\lambda[\delta x] = L_\lambda[x^{(1)}]-L_\lambda[x^{(2)}]=0$.
Corollary~\ref{cor:energyestimate}, applied to $\delta x$ with $\rho\equiv0$, gives the bound.
\end{proof}

\section{Conclusion and Future Directions}\label{sec:future}

The study of surfaces with negative curvature has revealed fundamental connections between geometry, analysis, and physical models of thin elastic sheets. In this work, we focused on the hyperbolic Monge--Ampère equation in strip-like geometries, motivated both by its geometric role as the prescribed Gauss curvature equation and by its analytic formulation of localized negative curvature that decays at infinity. We established well-posedness for the transformed Cauchy--Goursat problem in the hodograph weak class, including a parametrix--corrector construction at the singular corner. Reconstruction in physical variables is a further step, available for sufficiently regular solutions on regions where $p\mapsto x(u,p)$ remains injective with $x_p$ of fixed nonzero sign; the fold example shows why this qualification is necessary.

A particularly interesting feature of our problem is that the hodograph transformation reveals a geometric \textit{``arrow of time''}, in contrast to the standard undamped wave equation, whose dynamics are time-reversible. This asymmetry is dictated at the geometric level: if complete Cauchy data are prescribed on a ``near'' boundary (i.e., for large $u=s+t=C$) together with independent Goursat data on the axes, the Cauchy development already determines the traces that reach those axes. The data are therefore generically over-determined and may be inconsistent. In this existence and compatibility sense, the reversed mixed Cauchy--Goursat problem is ill posed. Prescribing the Cauchy data instead in the ``far field'' (i.e., for small $u=s+t=c$) avoids this overdetermination.

The results obtained in this article suggest several possible directions for future work:
\begin{itemize}

\item \textbf{Arrow of time.} The preceding analysis of compatibility and overdetermination is complemented by the energy analysis, which also picks out an asymmetry in the ``time'' direction. The energy identity of Section~\ref{sec:estimates} gives forward, i.e., in the direction $\lambda'(u) > 0$, quantitative control for the stated finite-energy class; in the singular-corner regime $\lambda(0)=0$ of \S\ref{sec:singularities}, this bound degenerates as $c \to 0$. These forward bounds, however, do not prove instability for a backward Cauchy problem, a natural conjecture here.

    \item \textbf{Criteria for Fold-Free Solutions and Stability.} A crucial next step is to find sufficient conditions on the Cauchy-Goursat data $\{g,h,f,n\}$ and the coefficient $\lambda(u)$ that guarantee a fold-free solution, defining the precise parameter space for physically realizable smooth surfaces. A related goal is to develop a nonlinear stability theory, showing that the property of being fold-free is preserved under small perturbations.

    \item \textbf{Sharper Regularity Theory.} A complementary direction is to sharpen the regularity theory for these solutions. This would involve investigating how stronger $H^k$-compatibility conditions on the boundary data can be propagated into the domain to obtain higher-order energy estimates and, under additional assumptions on the data, establish the existence of smooth $C^k$ solutions.

    \item \textbf{Self-Similar Solutions and Asymptotic States.} The semi-infinite cylinder case invites a search for self-similar solutions, which often act as attractors for a large class of initial data in hyperbolic systems. Finding such solutions and proving their stability would be a major step in understanding the universal, long-term (``large $y$") asymptotic geometry of these surfaces.

    \item \textbf{Numerical Implementation.} The parametrix-corrector decomposition is not just an analytical tool; it provides a powerful blueprint for numerical computation. Developing a numerical scheme based on this decomposition would allow for the efficient and accurate computation of solutions, as the parametrix analytically resolves the corner singularity, leaving a more regular problem for the corrector that is amenable to standard numerical methods.

    \item \textbf{Broader Generalizations.} Further extensions include allowing the curvature function $\lambda(x,y)$ to depend on both spatial variables, which would broaden the class of admissible geometries while posing new analytic challenges. Another avenue is the study of homogenization problems, where rapidly oscillating coefficients in $\lambda$ could lead to effective large-scale behavior with both theoretical and applied significance.

    \item \textbf{The Rigidity-Flexibility Transition.} Perhaps the most significant direction is to investigate the transition from the rigid to the flexible regime. This would involve studying whether solutions, when ``pushed", develop the branch point defects that characterize the intricately wrinkled surfaces seen in nature. Such an analysis would provide a direct, unifying link between the two sides of the rigidity-flexibility dichotomy.
\end{itemize}

\section*{Acknowledgments}

SCV was supported in part by NSF awards DMS-2108124 and DMS- 2511503. SCV gratefully acknowledges John Gemmer for valuable discussions.

\end{document}